\documentclass[letterpaper,11pt]{amsart}

\pdfoutput=1
\usepackage{tikz, tikz-cd}
\usetikzlibrary{arrows}
\usepackage[font=footnotesize,labelfont=bf]{caption}
\usepackage{blkarray}
\usepackage{enumitem}
\usepackage{hyperref}
\usepackage{latexsym,array,delarray,epsfig,setspace,cleveref, mathtools,amssymb,mathrsfs,bbold}
\usepackage{subfig}

\newtheorem{theorem}{Theorem}
\numberwithin{theorem}{section}
\newtheorem{proposition}[theorem]{Proposition}
\newtheorem{corollary}[theorem]{Corollary}
\newtheorem{lemma}[theorem]{Lemma}
\newtheorem{conjecture}[theorem]{Conjecture}
\newtheorem{question}[theorem]{Question}
\theoremstyle{definition}
\newtheorem{definition}[theorem]{Definition}
\newtheorem{remark}[theorem]{Remark}
\newtheorem{example}[theorem]{Example}
\newtheorem{problem}[theorem]{Problem}

\newcommand{\Q}{{\mathbb Q}}
\newcommand{\Z}{{\mathbb Z}}

\renewcommand{\k}{\mathbb{k}}
\renewcommand{\O}{\mathcal{O}}

\newcommand{\R}{\mathbb{R}}

\newcommand{\B}{\mathcal{B}}
\newcommand{\BB}{\mathbb{B}}
\newcommand{\F}{\mathcal{F}}

\newcommand{\M}{\mathcal{M}}
\newcommand{\E}{\mathcal{E}}
\newcommand{\EE}{\mathfrak{E}}
\newcommand{\FF}{\mathfrak{F}}

\newcommand{\wt}{\textup{wt}}

\renewcommand{\P}{{\mathbb P}}

\newcommand{\vv}{\mathfrak{v}}
\newcommand{\bv}{{\bf v}}
\newcommand{\e}{\textbf{e}}
\newcommand{\br}{\textbf{r}}
\newcommand{\bs}{\textbf{s}}

\DeclareMathOperator{\Trop}{Trop}
\DeclareMathOperator{\Hom}{Hom}

\DeclareMathOperator{\GL}{GL}

\DeclareMathOperator{\In}{in}
\DeclareMathOperator{\rank}{rank}
\DeclareMathOperator{\PL}{PL}
\DeclareMathOperator{\Berg}{Berg}
\DeclareMathOperator{\Span}{span}
\DeclareMathOperator{\conv}{conv}
\DeclareMathOperator{\ch}{ch}
\DeclareMathOperator{\codim}{codim}

\newcommand{\GF}{\textup{GF}}

\title{Tropical vector bundles and matroids}
\author{Kiumars Kaveh}
\author{Christopher Manon}

\begin{document}

\begin{abstract}
We introduce a notion of tropical vector bundle on a tropical toric variety which is a tropical analogue of a torus equivariant vector bundle on a toric variety. Alternatively it can be called a toric matroid bundle. We define equivariant $K$-theory and characteristic classes of these bundles. As a particular case, we show that any matroid comes with tautological tropical toric vector bundles over the permutahedral toric variety and the corresponding equivariant $K$-classes and Chern classes recover the tautological classes of matroids constructed in \cite{BEST}. In analogy with toric vector bundles, we define sheaf of sections and Euler characteristic as well as positivity notions such as global generation, ampleness and nefness for tropical toric vector bundles. Moreover, we prove a vanishing of higher cohomologies result. Finally, we study the splitting of our tropical toric vector bundles and, in particular, an analogue of Grothendieck's theorem on splitting of vector bundles on $\mathbb{P}^1$. 
\end{abstract}

\maketitle

\tableofcontents

\section{Introduction}
The theory of line bundles/divisors on tropical varieties is a very well studied subject and various aspects of it are well developed. On the other hand, the theory of higher rank vector bundles on tropical varieties is far less developed and a good general framework has been lacking. One of the early works on tropical vector bundles is  \cite{Allermann}. Recently, tropical vector bundles have been introduced and studied on metric graphs in \cite{Ulirsch}. Also, a notion of a tropical vector bundle has been explored by Jun, Mincheva and Tolliver (\cite{Kalina}) from the point of view of tropical scheme theory. In their theory, vector bundles are always direct sums of line bundles.

In this paper, we introduce a notion of a tropical vector bundle over a tropical toric variety. When the matroid data in the definition is representable, these are tropicalizations of toric vector bundles (torus equivariant vector bundles on toric varieties). 

Toric vector bundles of rank $r$ were famously classified by Klyachko (in the remarkable paper \cite{Klyachko}) in terms of compatible systems of filtrations on an $r$-dimensional vector space (the first classification of toric vector bundles goes back to Kaneyma \cite{Kaneyama}).
In \cite{Kaveh-Manon-Building, Kaveh-Manon-TVBs-valuation}, the authors interpret the Klyachko data of a rank $r$ toric vector bundle $\E$ as an \emph{integral piecewise linear map} from the fan $\Sigma$ of the toric variety $X_\Sigma$ to the space $\tilde{\B}(E)$ of valuations on an $r$-dimensional vector space $E$. Then, using the language of piecewise linear maps, they reformulate the Klyachko data in the following two ways (see Section \ref{sec-prelim-toric-vb}):
\begin{itemize}
    \item[(a)] As an \emph{integral piecewise linear map} $\Phi_\E$ from $\Sigma$ to the tropical variety of a linear ideal $L$. Conversely, any such integral piecewise linear map gives rise to a toric vector bundle on $X_\Sigma$ (Section \ref{subsec-tvb-PL-map-trop}).  
    \item[(b)] As a tropical point of a linear tropical ideal, valued over the semifield of integral piecewise linear functions. Conversely, any such point gives rise to a toric vector bundle (Section \ref{subsec-tvb-PL-valuation}).
\end{itemize}

In the present paper we observe that the notions in (a) and (b) above make sense for an arbitrary (not necessarily representable) matroid $\M$. We call the corresponding (equivalent) combinatorial objects \emph{tropical toric vector bundles}.

\begin{definition}[Tropical toric vector bundle]  \label{def-TMB-intro}
Let $\Sigma$ be a fan and $\M$ a matroid. 
By a \emph{tropical toric vector bundle} $\EE$ over $X_\Sigma$ and with matroid $\M$, we mean the following equivalent data:
\begin{itemize}
    \item[(a)] A \emph{piecewise linear map} $\Phi=\Phi_\EE: |\Sigma| \to \Berg(\M)$ (Definition \ref{def-matroid-vb-1}).
    \item[(b)] A \emph{tropical point} $\vv = \vv_\EE$ on $\Berg(\M)$, valued in the semifield of integral piecewise linear functions, that satisfies a compatibility condition with $\Sigma$ (see Definition \ref{def-matroid-vb-2}). We point out that any tropical point on $\Berg(\M)$ is compatible with a sufficiently refined fan $\Sigma$.  
\end{itemize}
\end{definition}

One can pull-back a tropical toric vector bundle to a subfan. We regard this as a tropical vector bundle on a tropical subscheme.

\begin{remark}
We are tempted to think of tropical toric vector bundles as \emph{toric vector bundles over $\mathbb{F}_1$} (the field with one element).
\end{remark}

\begin{remark}  \label{rem-tropical-vb}
In the first version of the paper we used the name \emph{toric matroid bundle} for what we now call a  \emph{tropical toric vector bundle}. As we were working on the first version, we became aware of the overlap of our work with the work of Khan and Maclagan (\cite{Khan-Maclagan}). In particular, in this work, the authors independently introduce (basically) the same object as in this paper (we actually coordinated so that the two papers appeared on the arXiv on the same day, see also Remark \ref{rem-Khan-Maclagan} for more details). In part, we decided to make the name change to tropical toric vector bundle to agree with \cite{Khan-Maclagan}. 
\end{remark}

As mentioned above, a toric vector bundle can be described by the data $(L, \Phi)$ where $L$ is a linear ideal in a polynomial ring $\k[x_1, \ldots, x_m]$. The process of tropicalization replaces $L$ with the matroid $\M(L)$ induced by $L$ on the images of the generators in the quotient $\k[x_1, \ldots, x_m]/L$. In particular, the Bergman fan of this matroid is the tropicalized linear space $\Trop(L)$. From this point of view, as in the tropical lifting problem for tropical linear spaces, the tropical lifting problem for tropical toric vector bundles hinges on the representability of the matroid $\M$. From now on we refer to the tropical toric vector bundle $(\M(L), \Phi)$ obtained from the toric vector bundle $\E$ associated to the data $(L, \Phi)$ as the \emph{tropicalization} of $\E$.
We will see that for a fixed toric vector bundle $\E$, the information preserved by the tropicalization depends on the choice of data $(L, \Phi)$. For any such choice, the characteristic classes of $\E$ will coincide with those of the tropicalization $\EE = (\M(L), \Phi)$ (see Proposition \ref{prop-tropical-characteristic-comparison}). However, other data such as the matroid of global sections will only recover the global sections of $\E$ if $\M(L)$ is a \emph{DJS matroid} of $\E$ (see Definition \ref{def-DJSmatroid-toricvectorbundle} and Proposition \ref{prop-Euler-sheaf-tropicalization}).

\begin{remark}
When the matroid is representable, these tropical toric vector bundles can be obtained as the tropicalizations of toric vector bundles with respect to a natural set of generators for the corresponding Cox module (see Sections \ref{subsec-review-sheaf} and \ref{subsec-sheaf} as well as \cite{Khan-Maclagan}).    
\end{remark}

\begin{remark}
In a future work, we will extend the definition of a tropical toric vector bundle to the non-trivially valued case. Namely, we will consider tropical toric vector bundles on toric schemes over a discrete valuation ring and valuated matroid.      
\end{remark}

{\bf Characteristic classes and tautological bundles.} Motivated by the construction of equivariant Chern classes of toric vector bundles (\cite[Proposition 3.1]{Payne-moduli} and \cite[Corollary 3.5]{Kaveh-Manon-Building}) we introduce the equivariant K-class $[\EE]$ and equivariant Chern classes $c_i^T(\EE)$ of a tropical toric vector bundle $\EE$. These are piecewise exponential (respectively polynomial) functions obtained by composing the piecewise linear map $\Phi_\EE$ with the universal exponential function (respectively universal elementary symmetric functions) on the Bergman fan (Section \ref{sec-K-class-char-class}).  When $\EE$ is obtained from the data of a toric vector bundle $\E$, the Chern classes and $K$-classes of $\EE$ and $\E$ coincide. 

One of our motivations for the study of tropical toric vector bundles is to provide a natural setting for the work \cite{BEST} on tautological classes of matroids. As shown in \cite{BEST}, these tautological classes encode the information of Tutte polynomial of the matroid. Let $X_m$ denote the toric variety associated to the permutahedral fan corresponding to permutations of $m$ indices. In \cite{BEST} two toric vector bundles are built from the data of a representable matroid $\M(L)$: the universal subbundle $\mathcal{S}_{\M(L)}$, and the universal quotient bundle $\mathcal{Q}_{\M(L)}$. Positivity properties of $\mathcal{Q}_{\M(L)}$ and the dual bundle of $\mathcal{S}_{\M(L)}$ allow important properties of matroid invariants, such as a log-concavity statement related to the coefficients of the Tutte polynomial, to be recovered from the intersection theory of the characteristic classes of these bundles.  These properties are then extended to all matroids using the theory of valuative invariants. Finding the correct geometric object underlying these characteristic classes in the non-representable case was a motivational question for the work in this paper. Let $\phi_\chi: X_m \to X_m$ be the Cremona transformation. We show the following (Section \ref{sec-tautological-matroid-bundles}):
\begin{theorem}\label{thm-main-tautological}
To any matroid $\M$ with $m = |\M|$ there naturally corresponds a \emph{tautological tropical toric vector bundle} $\EE_\M$ on $X_m$. In the case of a matroid $\M(L)$ associated to a linear ideal $L$, $\EE_{\M(L)}$ is { the tropicalization of} the dual of the universal subbundle, and $\phi^*_\chi\EE_{\M^\vee(L)}$ is { the tropicalization of} the universal quotient bundle (coming from embedding in a Grassmannian). Moreover, the equivariant tautological classes introduced in \cite{BEST} coincide with the Chern classes of $\EE_\M$ and $\phi^*_\chi\EE_{\M^\vee}$.  \end{theorem}

{\bf Sheaf of sections and positivity.}
In Section \ref{subsec-sheaf} we introduce the notion of \emph{sheaf of sections} of a tropical toric vector bundle $\EE$. It extends the geometric notion of the sheaf of sections of a toric vector bundle. To each toric open chart $U_\sigma$, $\sigma \in \Sigma$, the sheaf of sections of $\EE$ assigns a certain matroid. This leads us to define the notion of a \emph{globally generated tropical toric vector bundle} which extends that of toric vector bundles. In particular, for each character $u $, we have the notions of rank of space of global $u$-weight sections $H^0(X_\Sigma, \EE)_u$. The \emph{parliament of polytopes}, introduced by DiRocco, Jabbusch, and Smith in \cite{DJS}, captures global generation of toric vector bundles.  We show this notion makes sense for tropical toric vector bundles and use it to give a criterion for the global generation of a tropical toric vector bundle extending that of \cite{DJS} (Theorem \ref{thm-gg}). 

We expect the Chern classes of a globally generated tropical toric vector bundle to have similar positivity and log-concavity properties as those of a globally generated vector bundle.  For example, the following would generalize the Khovanskii-Teissier inequalities \cite[Theorem 1.6.1]{Lazarsfeld} (see \cite[Theorem 1.3]{RossToma} for the geometric analogue).  

\begin{conjecture}
Let $\EE$ be a globally generated tropical toric vector bundle of rank $r$ on a smooth, projective toric variety $X_\Sigma$ of dimension $d \geq r$, and let $\alpha$ be the Chern class of an ample line bundle on $X_\Sigma$. Then the intersection numbers $c_i(\EE)\alpha^{d-i}$ form a log-concave sequence. 
\end{conjecture}

We also introduce tentative notions of \emph{ample} and \emph{nef} tropical toric vector bundles following \cite{HMP} (see Definition \ref{def-ample-nef-tmb}). More precisely, we say that a tropical toric vector bundle $\EE$ is \emph{ample} (respectively \emph{nef}) if its restriction to any $1$-dimensional torus orbit closure in $X_\Sigma$ is split and is equivalent to a sum of ample (respectively nef) line bundles (in the sense of Definition \ref{def-tmbec1} and Definition \ref{def-tmbec2}, see also below for equivalence and splitting). This definition remains tentative because it is not yet known, at least to the authors, if splitting of tropical toric vector bundles over $\P^1$ works in the same way as in the case of vector bundles (see Section \ref{sec-splitting}). Regardless, in Sections \ref{sec-tautological-matroid-bundles} and \ref{sec-splitting} we show the following expected facts (Corollary \ref{cor-tautologicalGG} and Theorem \ref{thm-neftautological}): 

\begin{theorem}
The tautological bundle $\EE_\M$ is globally generated and nef.     
\end{theorem}

In Section \ref{subsec-euler}, we introduce the equivariant Euler characteristic $\chi(X_\Sigma, \EE)_u$ of a tropical toric vector bundle $\EE$, where $u$ is any character of the torus $T$.  
Using the results of Khovanskii and Pukhlikov (\cite{Khovanskii-Pukhlikov-1,Khovanskii-Pukhlikov-2}), one can show that, as in the case of geometric vector bundles, the function $\mathcal{L} \mapsto \chi(X_\Sigma, \EE \otimes \mathcal{L})$ is a polynomial on the Picard group of $X_\Sigma$. 

For any character $u$, let $h^0(X_\Sigma, \EE)_u$ denote the rank of the global section matroid $H^0(X_\Sigma, \E)_u$. In Section \ref{subsec-euler} we prove the following. 
\begin{theorem}[A vanishing of higher cohomologies]  \label{thm-intro-sections}
    Let $\EE$ be a tropical toric vector bundle on a smooth, projective toric variety $X_\Sigma$, and let $\mathcal{L}$ be an ample line bundle on $X_\Sigma$, then there is an integer $N_0 >0$ such that for all $N \geq N_0$ and $u \in M$ we have:\[\chi(X_\Sigma,\EE\otimes\mathcal{L}^{\otimes N})_u = h^0(X_\Sigma,\EE\otimes\mathcal{L}^{\otimes N})_u.\]
    Moreover, for $N \geq N_0$, $h^0(X_\Sigma, \E\otimes \mathcal{L}^{\otimes N})$ is computed by an integral polynomial of degree $d=\dim(X_\Sigma)$ in $N$.
\end{theorem}
The proof of Theorem \ref{thm-intro-sections} in fact gives an effective bound for $N_0$ involving polyhedra associated to the flats of $\M$. These polyhedra can be thought of as higher members of the parliament of polytopes of DiRocco, Jabbusch, and Smith \cite{DJS}. An example of the above theorem is the following. Let $-K_m$ be the anticanonical class of the permutahedral variety $X_m$, and let $\M$ be a matroid with $m$ elements. Then the function $\ell \mapsto h^0(X_m, \EE_\M\otimes \O(-\ell K_m))$ is a polynomial for sufficiently large $\ell$.


Although Theorem \ref{thm-intro-sections} suggests vanishing of ``higher cohomology'', as of now, we do not know how to define higher cohomologies for tropical toric vector bundles.  Our methods are combinatorial, utilizing the theory of convex chains introduced by Khovanskii and Pukhlikov (\cite{Khovanskii-Pukhlikov-1,Khovanskii-Pukhlikov-2}).  When the tropical toric vector bundle $\EE$ is obtained from the data of a toric vector bundle $\E$, the functions which count dimensions of global sections will coincide provided the toric vector bundle is defined using a matroid of DiRocco, Jabbusch, and Smith ( \cite{DJS}, \cite[Section 3.4]{Kaveh-Manon-TVBs-valuation}).

We expect that the higher cohomologies of tautological bundle $\EE_\M$ of a matroid $\M$ vanish. 

\begin{problem}  \label{prob-tautologicalpoly}
For a matroid $\M$ with tautological bundle $\EE_\M$ on the permutahedral variety $X_m$, do $h^0(X_m, \EE_\M)$ and $\chi(X_m, \EE_\M)$ have combinatorial meaning? More generally, does the polynomial $\mathcal{L} \mapsto \chi(\EE_\M \otimes \mathcal{L})$, on the Picard group of the permutahedral variety, have a combinatorial meaning? 
\end{problem}
By results in Section \ref{sec-Euler-char-matroid-vb}, computation of the above polynomial should involve the Ehrhart theory of various permutahedra.

\begin{problem}
Given a matroid $\M$, give a combinatorial description of the set of tropical toric vector bundles on permutahedral toric variety with matroid $\M$ whose higher cohomologies vanish.   \end{problem}

For further study, we also pose the following:
\begin{problem}  \label{prob-HRR-matroid-vb}
For a tropical toric vector bundle $\EE$, formulate a Hirzebruch-Riemann-Roch formula relating the Chern classes of $\EE$ and $\chi(X_\Sigma,\EE)$. \end{problem}

\begin{remark}
In fact, one can associate to each tropical toric vector bundle $\EE$ a convex chain in the sense of \cite{Khovanskii-Pukhlikov-1}. The Euler characteristic of $\EE$ is then the sum of values of this convex chain on lattice points. We expect that the Khovanskii-Pukhlikov Riemann-Roch theorem (\cite{Khovanskii-Pukhlikov-2}) gives an answer to Problem \ref{prob-HRR-matroid-vb}. This is a work in progress.   
\end{remark}

{\bf Matroid extension.} In Section \ref{sec-matroid-extension} we consider tropical toric vector bundles up to extension of matroids.
We recall (Section \ref{subsec-tvb-PL-map-trop} and \cite[Section 4]{Kaveh-Manon-TVBs-valuation}) 
that different linear ideals $L$, and hence different linear matroids, can give rise to the same toric vector bundle on $X_\Sigma$. In particular, if we enlarge the spanning set $\M$ (the ground set of our linear matroid) it gives rise to the same toric vector bundle. This observation motivates considering the notion of a \emph{matroid extension} and study tropical toric vector bundles up to extension of matroids.   

Let $\M_1$ and $\M_2$ be (the ground sets of) matroids of equal rank $r$. We say that a one-to-one map on the underlying sets $\phi:\M_1 \to \M_2$ is a \emph{matroid extension} if the matroid induced on $\phi(\M_1)$ by $\M_2$ is $\M_1$. For a fixed extension $\phi: \M_1 \to \M_2$, a tropical toric vector bundle $\EE$ on $X_\Sigma$ induces a tropical toric vector bundle $\phi_*\EE$ with matroid $\M_2$.  
In the case of representable matroids, this operation corresponds to ``enlarging'' a spanning set of vectors in a vector space, and gives rise to isomorphic toric vector bundles.  

{\bf Splitting.} In the last section, we study splitting of tropical toric vector bundles. An important question in geometry of vector bundles is when a vector bundle can be decomposed into a sum of line bundles. We say that a vector bundle is \emph{split} if it is isomorphic to a sum of line bundles. A celebrated theorem of Grothendieck states that any vector bundle over $\mathbb{P}^1$ splits. 

A toric vector bundle is \emph{equivariantly split} if it is equivariantly isomorphic to a sum of toric line bundles. It can be shown that a toric vector bundle $\E$, with piecewise linear map $\Phi$, is split if and only if the image of $\Phi$ lands in a single apartment (see Definition \ref{def-space-of-val}). Motivated by this we make the following definition.
\begin{definition}
A tropical toric vector bundle $(\M, \Phi)$ is \emph{split} if the image of $\Phi$ lies in a single apartment of $\Berg(\M)$ (see Definition \ref{def-apartment-Bergman-fan}).  Equivalently, there is a basis $B \subset \M$ such that, for any $e \in \M$, $\vv(e) = \min\{\vv(c) \mid c \in C\cap B\}$, where $C \subset \M$ is the unique circuit with $C \setminus \{e\} \subset B$. 
\end{definition}

Strictly speaking, a split tropical toric vector bundle is not isomorphic to a sum of toric line bundles. But it is the case if we consider the extension class of a tropical toric vector bundle. More precisely, let $(\M, \vv)$ be a tropical toric vector bundle. One verifies that the following are equivalent: 
\begin{itemize}
\item[(i)] The class $[(\M, \vv)]$ contains a pair $(M', \vv')$ which is split.     
\item[(ii)] The class $[(\M, \vv)]$ contains a member of the form $(B, \vv)$, where $B$ is a single basis.
\item[(iii)] The class $[(\M, \vv)]$ contains the pair associated to a direct sum of toric line bundles.      
\end{itemize}

Finally, we address the question of splitting of tropical toric vector bundles over $\mathbb{P}^1$.
The equivariant version of Grothendieck's theorem for toric vector bundles is an immediate  corollary of the linear algebra fact that any two flags of subspaces in a finite dimensional vector space are adapted to the same basis. This fact is indeed one of the \emph{building} axioms for the Tits building of the general linear group. Namely, any two simplices in a building lie in the same apartment. Motivated by this we ask the following.

\begin{question}
Does the space of rank $r$ matroids in the same extension class have properties analogous to a building (that is, satisfies analogues of the building axioms)? 
\end{question}

We say that a matroid $\M$ is \emph{modular} if the submodular inequality for the rank function is an equality. We show the following (Corollaries \ref{cor-split-P1-modular} and \ref{cor-tmb-P1-no-split}).
\begin{theorem}
    Suppose $\M$ can be extended to a (possibly infinite) modular matroid $\mathcal{N}$, then any tropical toric vector bundle class $[(\M, \Phi)]$ over $\P^1$ splits. 
\end{theorem}

\begin{theorem}\label{thm-vamos-intro}
There is a bundle $(M, \Phi)$ over $\P^1$ which cannot be extended to a split bundle. 
\end{theorem}

\begin{remark}  \label{rem-Khan-Maclagan}
Here are a few more comments about the overlaps and differences with the work of Khan and Maclagan (\cite{Khan-Maclagan}): in the present paper we work with matroids while Khan and Maclagan work with valuated matroids. 
Similar to the present paper, they also have a notion of Cox module (Section \ref{subsec-sheaf}). They introduce a notion of \emph{(semi)stability} for tropical toric vector bundles which we do not address here. For this they define the first Chern class of such a bundle which we expect to coincide with ours (more generally, we define equivariant total Chern classes and equivariant K-classes). As we do in the present paper, they also address the question of splitting and in particular, splitting over $\mathbb{P}^1$. Likewise, they also realize the relevance of notion of \emph{modularity} of a matroid to this problem.  
\end{remark}

\bigskip
\noindent{\bf Acknowledgements:}
C. Manon was partially supported by Simons Collaboration Grant 587209 and National Science Foundation grant DMS-2101911. K. Kaveh was partially supported by National Science Foundation grant DMS-2101843 and a Simons Collaboration Grant for Mathematicians. We thank University of Bath and Goethe University Frankfurt for their hospitality during the workshops Combinatorial algebraic geometry (Bath, 2022) and Vector bundles and combinatorial algebraic geometry (Frankfurt, 2023). We also thank Diane Maclagan, Bivas Khan, Martin Ulirsch, Felipe Rinc\'on and Alex Fink for useful conversations. We thank Chris Eur for conversations concerning the Euler characteristic which lead to Theorem \ref{thm-intro-sections}. We thank Austin Alderete for conversations about the Vamos matroid which lead to Theorem \ref{thm-vamos-intro}.
\bigskip

\noindent{\bf Notation:}
\begin{itemize}
\item $\k$ denotes the ground field.
\item $T \cong \mathbb{G}_m^d$ denotes a (split) algebraic torus with $M$ and $N$ its character and cocharacter lattices respectively. In general, $M$ and $N$ denote rank $d$ free abelian groups dual to each other. We denote the pairing between them by $\langle \cdot, \cdot \rangle: M \times N \to \Z$. We let $M_\R = M \otimes \R$ and $N_\R = N \otimes \R$ be the corresponding $\R$-vector spaces. 
\item $U_\sigma$ is the affine toric variety corresponding to a (strictly convex rational polyhedral) cone $\sigma \subset N_\R$.
\item $\Sigma$ is a fan in $N_\R$ with corresponding toric variety $X_\Sigma$. We denote the support of $\Sigma$, i.e. the union of cones in it, by $|\Sigma|$. 
\item We fix a point $x_0$ in the open orbit $U_0$ in the toric variety $X_\Sigma$. The choice of $x_0$ identifies $U_0$ with the torus $T$.
\item $\PL(N_\R, \R)$ is the set of piecewise linear functions on the $\R$-vector space $N_\R$. We denote the subset of piecewise linear functions that attain integer values on $N$ by $\PL(N, \Z)$. Finally $\PL(\Sigma, \R)$ (respectively $\PL(\Sigma, \Z)$) denotes the subset of piecewise linear functions (respectively integral piecewise linear functions) that are linear on cones in $\Sigma$. 
\item $\mathcal{E}$ denotes a rank $r$ toric vector bundle on a toric variety $X_\Sigma$.
\item $E$ is an $r$-dimensional vector space which we usually think of as the fiber $\mathcal{E}_{x_0}$ of a rank $r$ toric vector bundle $\mathcal{E}$. 
\item $\Trop(I)$ is the tropical variety of an ideal $I$.
\item $\M$ denotes a matroid with ground set $\{1, \ldots, m\}$. By abuse of notation we use $\M$ to denote the ground set of the matroid as well. We use $r$ to denote the rank of $\M$. (We caution that unlike much of the matroid literature we do not denote the ground set by $E$, rather $E$ in this paper denotes an $r$-dimensional vector space. This is consistent with notation used in toric vector bundle literature.) 
\item $\Berg(\M)$ denotes the Bergman fan of $\M$. 
\item $\GF(\M)$ denotes the Gr\"obner fan of $\M$.
\item $P(\M)$ denotes the matroid polytope of $\M$.
\item $\EE$ denotes a tropical toric vector bundle over $X_\Sigma$ with matroid $\M$. We represent $\EE$ by pairs $(\M, \Phi)$ (where $\Phi$ is an integral piecewise linear map to $\Berg(\M)$) or $(\M, \vv)$ (where $\vv$ is a $\PL(N, \Z)$-valued tropical point on $\Berg(\M)$).
\item $[\EE]$ and $c_i^T(\EE)$ denote the equivariant $K$-class and equivariant Chern classes of a tropical toric vector bundle $\EE$.
\item $[\M]$ denotes the matroid extension class of a matroid $\M$.
\item $[(\M, \Phi)]$ (respectively $[(\M, \vv)]$) denotes the extension class of a tropical toric vector bundle $(\M, \Phi)$ (respectively $(\M, \vv)$).
\end{itemize}
\bigskip

\section{Preliminaries on toric vector bundles} \label{sec-prelim-toric-vb}
In this section we review some background material on toric vector bundles. 
\subsection{Klyachko classification of toric vector bundles}   \label{subsec-prelim-toric-vb}
Let $T \cong \mathbb{G}_m^d$ denote an $d$-dimensional (split) algebraic torus over a field $\k$. We let $M$ and $N$ denote its character and cocharacter lattices respectively. We also denote by $M_\R$ and $N_\R$ the $\R$-vector spaces spanned by $M$ and $N$. For cone $\sigma \in N_\R$ let $M_\sigma$ be the quotient lattice:
$$M_\sigma = M / (\sigma^\perp \cap M).$$
Let $\Sigma$ be a (finite rational polyhedral) fan in $N_\R$ and let $X_\Sigma$ be the corresponding toric variety. We let $U_\sigma$ denote the invariant affine open subset in $X_\Sigma$ corresponding to a cone $\sigma \in \Sigma$. We denote the support of $\Sigma$, that is the union of all the cones in $\Sigma$, by $|\Sigma|$. For each $i$, $\Sigma(i)$ is the subset of $i$-dimensional cones in $\Sigma$. In particular, $\Sigma(1)$ is the set of rays in $\Sigma$. For each ray $\rho \in \Sigma(1)$ we let ${\bf v}_\rho$ be the primitive vector along $\rho$.

We say that $\E$ is a {\it toric vector bundle} on $X_\Sigma$ if $\E$ is a vector bundle on $X_\Sigma$ equipped with a linear action of $T$ on $\E$ which lifts the $T$-action on $X_\Sigma$. 
We fix a point $x_0 \in X_0 \subset X_\Sigma$ in the dense orbit $U_0$. We often identify $U_0$ with $T$ and think of $x_0$ as the identity element in $T$. We let $E = \E_{x_0}$ denote the fiber of $\E$ over $x_0$. It is an $r$-dimensional vector space where $r = \rank(\E)$. 

For each cone $\sigma \in \Sigma$, with invariant open subset $U_\sigma \subset X_\Sigma$, the space of sections $\Gamma(U_\sigma, \E\!\!\mid_{U_\sigma})$ is a $T$-module. We let  $\Gamma(U_\sigma, \E\!\!\mid_{U_\sigma})_u \subseteq \Gamma(U_\sigma, \E\!\!\mid_{U_\sigma})$ be the weight space corresponding to a weight $u \in M$. One has the weight decomposition: 

$$\Gamma(U_\sigma, \E\!\!\mid_{U_\sigma}) = \bigoplus_{u \in M} \Gamma(U_\sigma, \E\!\!\mid_{U_\sigma})_u.$$ Every section in $\Gamma(U_\sigma, \E\!\!\mid_{U_\sigma})_u$ is determined by its value at $x_0$.  Thus, by restricting sections to $E = \E_{x_0}$, we get an embedding $\Gamma(U_\sigma, \E_{|U_\sigma})_u \hookrightarrow E$. Let us denote the image of $\Gamma(U_\sigma, \E\!\!\mid_{U_\sigma})_u$ in $E$ by $E_u^\sigma$. Note that if $u' \in \sigma^\vee \cap M$ then multiplication by the character $\chi^{u'}$ gives an injection $\Gamma(U_\sigma, \E\!\!\mid_{U_\sigma})_u \hookrightarrow \Gamma(U_\sigma, \E\!\!\mid_{U_\sigma})_{u-u'}$. Moreover, the multiplication map by $\chi^{u'}$ commutes with the evaluation at $x_0$ and hence induces an inclusion $E_u^\sigma \subset E_{u-u'}^\sigma$. If $u' \in \sigma^\perp$ then these maps are isomorphisms and thus $E_u^\sigma$ depends only on the class $[u] \in M_\sigma = M/(\sigma^\perp \cap M)$. For a ray $\rho \in \Sigma(1)$ we write $$E^\rho_i = E_u^\rho,$$ for any $u \in M$ with $\langle u, {\bf v}_\rho \rangle = i$ (all such $u$ define the same class in $M_\rho$). Equivalently, one can define $E^\rho_u$ as follows (see \cite[\S 0.1]{Klyachko}). Pick a point $x_\rho$ in the orbit $O_\rho$ and let:
$$E^\rho_u = \{ e \in E \mid \lim_{t \cdot x_0 \to x_\rho} \chi^u(t)^{-1}(t \cdot e) \text{ exists in } \E \},$$
where $t$ varies in $T$ in such a way that $t \cdot x_0$ approaches $x_\rho$. 
We thus have a decreasing filtration of $E$:
\begin{equation}  \label{equ-filt-E-rho}
\cdots \supset E^\rho_{i-1} \supset E^\rho_i \supset E^\rho_{i+1} \supset \cdots
\end{equation}

An important step in the classification of toric vector bundles is that a toric vector bundle over an affine toric variety is {\it equivariantly trivial}. That is, it decomposes $T$-equivariantly as a sum of trivial line bundles (see \cite[Proposition 2.1.1]{Klyachko}). 

For $\sigma \in \Sigma$, we let $u(\sigma) = \{ [u_1], \ldots, [u_r]\} \subset M_\sigma$ be the multiset of characters by which $T$ acts on the trivial bundle $\E_{|U_\sigma}$. One then observes that, for each $\sigma \in \Sigma$, the filtrations $(E^\rho_i)_{i \in \Z}$, $\rho \in \Sigma(1)$, satisfy the following compatibility condition: 
There is a decomposition of $E$ into a direct sum of $1$-dimensional subspaces indexed by a finite multiset $u(\sigma) \subset M_\sigma$:
$$E = \bigoplus_{[u] \in u(\sigma)} L_{[u]},$$
such that for any ray $\rho \in \sigma(1)$ we have:
\begin{equation}  \label{equ-Klyachko-comp-condition}
E^\rho_i = \sum_{\langle u, {\bf v}_\rho \rangle \geq i}  L_{[u]}
\end{equation}

\begin{definition}[Compatible collection of filtrations]
We call a collection of decreasing $\Z$-filtrations $\{(E_i^\rho)_{i \in \Z} \mid \rho \in \Sigma(1) \}$ satisfying condition \eqref{equ-Klyachko-comp-condition} a \emph{compatible collection of filtrations}. 
(Moreover, for each $\rho$, we assume $\bigcap_{i \in \Z} E^\rho_i = \{0\}$ and $\bigcup_{i \in \Z} E^\rho_i = E$.)
\end{definition}

Let $E$, $E'$ be finite dimensional $\k$-vector spaces. Let $\{(E_i^\rho)_{i \in \Z} \mid \rho \in \Sigma(1) \}$ (respectively $\{({E'}_i^\rho)_{i \in \Z} \mid \rho \in \Sigma(1) \}$) be compatible collections of filtrations on $E$ (respectively $E'$). We say that a linear map $F: E \to E'$ is a \emph{morphism} from $\{(E_i^\rho)_{i \in \Z} \mid \rho \in \Sigma(1) \}$ to $\{({E'}_i^\rho)_{i \in \Z} \mid \rho \in \Sigma(1) \}$ if for every $\rho \in \Sigma(1)$ and $i \in \Z$ we have $F(E_i^\rho) \subset {E'}_i^\rho$. 
With this notion of morphism, for a fixed fan $\Sigma$, the compatible collections of filtrations on finite dimensional $\k$-vector spaces form a category.

The following is Klyachko's theorem on the classification of toric vector bundles (\cite[Theorem 2.2.1]{Klyachko}). 
\begin{theorem}[Klyachko]   \label{th-Klyachko}
The category of toric vector bundles on $X_\Sigma$ is equivalent to the category of compatible filtrations on finite dimensional $\k$-vector spaces.
\end{theorem}

\subsection{Toric vector bundles as piecewise linear maps to space of valuations}  \label{subsec-PL-maps}
We start by recalling the definition of a real valued valuation on a vector space. We will then see how to interpret the Klyachko data of compatible filtrations, for a toric vector bundle $\E$ on $X_\Sigma$ as an (integral) \emph{piecewise linear map} $\Phi$ from $|\Sigma|$ to the space $\tilde{\B}(E)$ of all valuations on $E$. 
We remark that the piecewise linear map $\Phi$ is essentially contained in Payne’s observation in \cite{Payne-cover} that the Klyachko data of a toric vector bundle can be used to construct a filtration-valued function on $|\Sigma|$.
This is also a special case of the main result in \cite{Kaveh-Manon-Building} where torus equivariant principal $G$-bundles over $X_\Sigma$, where $G$ is a reductive algebraic group, are classified in terms of \emph{piecewise linear maps to the (extended) Tits building of $G$}. 

\begin{definition}[Vector space valuation]   \label{def-val}
Let $E$ be a finite dimensional $\k$-vector space.
We call a function $v: E \to \overline{\R} = \R \cup \{\infty \}$ a {\it vector space valuation} (or a \emph{valuation} for short) if the following hold:
\begin{itemize}
\item[(1)] For all $e \in E$ and $0 \neq c \in \k$ we have $v(ce) = v(e)$. 
\item[(2)](Non-Archimedean property) For all $e_1, e_2 \in E$, $v(e_1+e_2) \geq \min\{v(e_1), v(e_2)\}$.
\item[(3)] $v(e)=\infty$ if and only if $e = 0$.
\end{itemize}
We call a valuation $v$ {\it integral} if it attains only integer values, i.e. $v: E \to \overline{\Z}$. 
\end{definition}

\begin{remark}
(i) In commutative algebra the term valuation usually refers to a valuation on a ring or algebra. Throughout most of this paper, we will use the term valuation to mean a valuation on a vector space. 
(ii) In \cite[Section 2.1]{KKh-Annals} (and some other papers) the term \emph{prevaluation} is used for a valuation on a vector space (to distinguish it from valuations on rings). 
\end{remark}

The {\it value set} $v(E)$ of a valuation $v$ is the image of $E \setminus \{0\}$ under $v$. It is easy to verify that $|v(E)| \leq \dim(E)$ and hence $v(E)$ is finite. 
Each integral valuation $v$ on $E$ gives rise to a filtration $E_{v, \bullet} = (E_{v \geq a})_{a \in \Z}$ on $E$ by vector subspaces defined by: 
$$E_{v \geq a} = \{ e \in E \mid v(e) \geq a\}.$$
If $v(E) = \{a_1 > \cdots > a_k\}$ then we have a flag: $$F_{v, \bullet} = (\{0\} \subsetneqq F_1 \subsetneqq \cdots \subsetneqq F_k=E),$$ where $F_i = E_{v \geq a_i}$. We note that the valuation $v$ is uniquely determined by the flag $F_{v, \bullet}$ and the $k$-tuple $(a_1 > \cdots > a_k)$.
Conversely, a decreasing filtration $E_\bullet = (E_a)_{a \in \Z}$ such that 
\begin{equation} \label{equ-flitration-conditions}
\bigcap_{a \in \Z} E_a = \{0\}, \text{ and } \bigcup_{a \in \Z} E_a = E, 
\end{equation}
defines a valuation $v_{E_\bullet}$ by:
$$v_{E_\bullet}(e) = \max\{ a \in \Z \mid e \in E_a\},$$ for all $e \in E$. 
It is straightforward to verify that 
the assignments $v \mapsto E_{v, \bullet}$ and $v \mapsto (F_{v, \bullet}, (a_1 >  \cdots > a_k))$ give one-to-one correspondences between the following sets: 
\begin{itemize}
\item[(i)] The set of integral valuations $v: E \to \overline{\Z}$. 
\item[(ii)] The set of decreasing  $\Z$-filtrations $E_\bullet$ on $E$ satisfying \eqref{equ-flitration-conditions}.
\item[(iii)] The set of flags $F_\bullet = (\{0\} \subsetneqq F_1 \subsetneqq \cdots \subsetneqq F_k = E)$ together with tuples of integers $(a_1 > \cdots > a_k)$. 
\end{itemize}

Recall that a frame $L = \{L_1, \ldots, L_r\}$ for $E$ is a collection of $1$-dimensional subspaces $L_i$ such that $E = \bigoplus_{i=1}^r L_i$. We say that a valuation $v$ is \emph{adapted} to a frame $L$ if every subspace $E_{v \geq a}$ is a sum of some of the $L_i$. This is equivalent to the following: For any $e \in E$ let us write $e = \sum_i e_i$ where $e_i \in L_i$. Then:
\begin{equation}  \label{equ-val-min}
v(e) = \min\{ v(e_i) \mid i=1, \ldots, r \}.
\end{equation}
If a valuation $v$ is adapted to a frame $L$, then $v$ is determined by the $r$-tuple $(v(L_1), \ldots, v(L_r))$. Conversely, any $r$-typle $(a_1, \ldots, a_r) \in \R^r$ determines a unique valuation $v$ adapted to $L$ by requiring that $v(e_i) = a_i$, for all $i=1, \ldots, r$ and $0 \neq e_i \in L_i$. In other words, $v$ is given by $v(e) = \min\{a_i \mid e_i \neq 0\}$.

\begin{definition}[Space of valuations/extended Tits building]   \label{def-space-of-val}
We denote by $\tilde{\B}(E)$ the set of all $\R$-valued valuations $v: E \to \overline{\R}$. We also denote the set of all $\Z$-valued valuations on $E$ (that is, the set of integral valuations on $E$) by $\tilde{\B}_\Z(E)$. For a frame $L$, we denote the set of valuations adapted to $L$ by $\tilde{A}(L)$. Also we denote by $\tilde{A}_\Z(L)$ the set of $\Z$-valued valuations adapted to $L$. As discussed above, $\tilde{A}(L)$ (respectively $\tilde{A}_\Z(L)$) can be identified with $\R^r$ (respectively $\Z^r$). We refer to $\tilde{\B}(E)$ (respectively $\tilde{A}(L)$) as the \emph{extended Tits building of $E$} (respectively an \emph{(extended) apartment}).
\end{definition}
  
The above gives a convenient way to package the Klyachko data (of compatible filtrations) of a toric vector bundle as a piecewise linear map into the space of valuations. 

\begin{definition}[Piecewise linear map to space of valuations/extended Tits building]   \label{def-plm}
With notation as before, a map $\Phi: |\Sigma| \to \tilde{\B}(E)$ is a \emph{piecewise linear map} if the following hold: For any $\sigma \in \Sigma$, there is a frame $L$ for $E$ such that $\Phi(\sigma)$ lands in an (extended) apartment $\tilde{A}(L)$. Moreover, we require that the restriction $\Phi_{|\sigma}: \sigma \to \tilde{A}(L)$ to be linear, i.e. it is the restriction of a linear map from $N_\R$ to $\tilde{A}(L)$. We say that a piecewise linear map $\Phi$ is \emph{integral} if $\Phi$ sends lattice points to lattice points, i.e. $\Phi(N \cap |\Sigma|) \subset \tilde{\B}_\Z(E)$.
\end{definition}

In \cite{Kaveh-Manon-Building, Kaveh-Manon-TVBs-valuation}, the Klyachko classification of toric vector bundles (Theorem \ref{th-Klyachko}) is restated as follows:
\begin{theorem}[Classification of toric vector bundles in terms of piecewise linear maps]   \label{th-Klyachko-plm}
The category of toric vector bundles on $X_\Sigma$ is equivalent to the category of integral piecewise linear maps to $\tilde{\B}(E)$, for all finite dimensional $\k$-vector spaces $E$. 
\end{theorem}

\subsection{Toric vector bundles as piecewise linear maps to tropical linear spaces}  \label{subsec-tvb-PL-map-trop}
We begin by reviewing some basic facts as well as some observations about tropicalized linear spaces. 

Let $\M = \{e_1, \ldots, e_m\} \subset E$ be a spanning set. Let $L \subset \k[x_1, \ldots, x_m]$ denote the linear ideal of relations among the $e_i$. We denote by $\GF(L)$ and $\Trop(L)$ the Gr\"obner fan and tropical variety of $L$ respectively. We have $\Trop(L) \subset |\GF(L)| \subset \R^m$. Moreover, the tropical variety is the support of a subfan of the Gr\"obner fan.  

For below, we need a bit of notation. Let $\{\delta_1, \ldots, \delta_m\}$ be the standard basis for $\R^m$. For a subset $J \subset \{1, \ldots, m\}$ we put $\delta_J = \sum_{j \in J} \delta_j$. 


\begin{definition}[Matroid polytope]  \label{def-matroid-polytope}
We recall that the \emph{matroid polytope} $P_\M$ is the convex hull of $\{\delta_B \mid B \subset \M \textup{ is a vector space basis} \}$.
\end{definition}

Observe that $P_\M$ is a subset of the hyperplane $H_r = \{ \sum_{i =1}^m a_i\delta_i \mid \sum_{i =1}^m a_i = r\} \subset \R^m$, where $r$ is the rank of $\M$. 
The following descriptions of the Gr\"obner fan $\GF(L)$ and tropical variety of $\Trop(L)$ are well-known (see \cite[Section 4.1]{MS}):

\begin{theorem}[Gr\"obner fan and tropical variety of a linear ideal]   \label{th-GF-trop-linear}
With notation as above, we have the following:
\begin{itemize}
\item[(a)] The Gr\"obner fan $\GF(L)$ is the outer normal fan to the matroid polytope $P_\M$. 
\item[(b)] The maximal cones in $\GF(L)$ are in one-to-one correspondence with the vector space bases in $\M$. For a basis $B \subset \M$ we denote the corresponding maximal face by $\sigma_B$.
\item[(c)] The tropical variety of $L$ consists of tuples $w=(w_1, \ldots w_m) \in \R^m$ such that for any circuit $C$ in the matroid defined by $\M$, the minimum $\min\{w_i \mid i \in C \}$ is attained at least twice. In other words, the linear polynomials $\sum_{i \in C} x_i$, for all circuits $C \subset \M$, form a tropical basis for $L$. 
\item[(d)] The tropical variety $\Trop(L)$ has a natural fan structure given by flags of flats in $\M$ (see Definition \ref{def-Bergman-fan}).
\end{itemize}
\end{theorem}

Motivated by the notion of an apartment in the space of valuations/extended Tits building $\tilde{\B}(E)$, we make the following definition (cf. Section \ref{subsec-PL-maps}):
\begin{definition}[Apartment in $\Trop(L)$]  \label{def-apt-Trop-L}
Let $B \subset \M$ be a basis. In analogy with apartments in the space of valuations $\tilde{\B}(E)$ (Definition \ref{def-space-of-val}), we call the intersection $\Trop(L) \cap \sigma_B$ an \emph{apartment} in $\Trop(L)$ and denote it by $A_B$. \end{definition}

\begin{remark}
When the first draft of this paper was in preparation, we learned that our notion of apartment is not new and has already been introduced by Felipe Rinc\'{o}n under the name \emph{local tropical linear space} (\cite{Rincon}). Nevertheless, for the purposes of the present paper and to emphasize the connection with building theory, we prefer to use the term apartment.   
\end{remark}

The next proposition shows that each apartment can be identified with $\R^r$ in a piecewise linear way. We postpone the proof to later when we introduce the notion of apartment for arbitrary matroids (Proposition \ref{prop-apartment-matroid}).
\begin{proposition}[Apartments are copies of $\R^r$]  \label{prop-apt-linear-space}
With notation as above, let $A_B \subset \Trop(L)$ be an apartment corresponding to a basis $B \subset \M$. For $i \in \M$ let $C_i$ denote the circuit in $B \cup \{i\}$ containing $i$. Define the map $\phi_B: \R^B \to A_B$, $\phi_B(a) = w = (w_1, \ldots, w_m)$ where:
$$
w_i = 
\begin{cases}
a_i \quad \textup{for } e_i \in B\\
\min\{ a_j \mid e_j \in C_i \setminus \{e_i\} \} \quad \textup{for } e_i \notin B.
\end{cases}
$$
Then $\phi_B$ is a piecewise-linear bijection between $\R^r$ and $A_B$. Hence every apartment can be thought of as a copy of $\R^r$.
\end{proposition}
The map $\phi_B: \R^r \to \Trop(L)$ is a section to the map $\pi_B: \Trop(L) \to \R^r$ given by projection onto the components corresponding to elements in $B$.

Let $L_1$ denote the elements of $L$ of homogeneous degree $1$.
We now see that $\Trop(L)$ naturally sits in the space of valuations $\tilde{\B}(E)$, where $E = \bigoplus_{i =1}^m \k x_i /L_1$. Moreover $\tilde{\B}(E)$ can naturally be projected onto $\Trop(L)$. By the fundamental theorem of tropical geometry (\cite[Section 3.2]{MS}), for any valuation $v: E \to \overline{\R}$, the $m$-tuple $(v(e_1), \ldots, v(e_m))$ lies on $\Trop(L)$. Thus $v \mapsto (v(e_1), \ldots, v(e_m))$ gives us a map $p: \tilde{\B}(E) \to \Trop(L)$. Conversely, for any $w \in \Trop(L) \cap \Q^m$ we can find a valuation $v: E \to \overline{\Q}$ such that $v(e_i) = w_i$, for all $i=1, \ldots, m$. By continuity, we get a map $i: \Trop(L) \to \tilde{\B}(E)$ such that $p \circ i = \textup{id}$ and hence $i$ is an embedding. More precisely, for a basis $B=\{b_1, \ldots, b_r\} \subset \M$, we can explicitly describe the restriction of the map $i$ to $A_B$ and see that it gives and identification of $A_B$ and $\tilde{A}(B)$. 
One computes that $i_B:= i_{|A_B}: A_B \to \tilde{A}(B)$ is given as follows: for any $w \in \Trop(L)$, $i_B(w): E \to \overline{\R}$ is the valuation given by:
$$i_B(w)(\sum_i \lambda_i b_i) = \min\{w_i \mid \lambda_i \neq 0 \}.$$
We have the following commutative diagram:
$$
\begin{tikzcd}
\Trop(L) \arrow[d, hook, "i"] & A_B \arrow[l, hook] \arrow[d, "i_B"] \\
\tilde{\B}(E) & \tilde{A}(B) \arrow[l, hook] & \R^r \arrow[l, "\cong"] \arrow[lu, "\cong"]
\end{tikzcd}
$$

Finally, we give a definition of piecewise linear map $\Phi_L: |\Sigma| \to \Trop(L)$ in the same manner as before (Definition \ref{def-plm}).
\begin{definition}[Piecewise linear map to a tropical linear space]  \label{def-PL-trop-lin}
$\Phi_L: |\Sigma| \to \Trop(L)$ is \emph{piecewise linear} if for any $\sigma \in \Sigma$, there is a basis $B$ such that the image $\Phi_L(\sigma)$ lies in the apartment $A_B$ and the composition $\pi_B\circ \Phi_L\!\!\mid_{\sigma}: |\sigma| \to \R^r$ is a linear map. We say $\Phi_L$ is integral if $\Phi_L(|\sigma| \cap N) \subset \Z^m$. 
\end{definition}

The following is a straightforward corollary of Theorem \ref{th-Klyachko-plm}:
\begin{theorem}[Toric vector bundles as piecewise linear maps to tropical linear spaces]  \label{th-Klyachko-plm-trop} 
Let $\E$ be a toric vector bundle over a toric variety $X_\Sigma$ with corresponding piecewise linear map $\Phi: 
|\Sigma| \to \tilde{\B}(E)$. Then under the embedding $\Trop(L) \hookrightarrow \tilde{\B}(E)$, the map $\Phi$ gives a piecewise linear map $\Phi_L: |\Sigma| \to \Trop(L)$. Conversely, any piecewise linear map $\Phi_L: |\Sigma| \to \Trop(L)$
gives rise to a toric vector bundle on $X_\Sigma$.
\end{theorem}

Finally, we introduce an integral matrix called the diagram, which captures the data of a toric vector bundle.  We assume a fixed bijection between $[n]:=\{1, \ldots, n\}$ and the rays $\Sigma(1)$. 

\begin{definition}[Diagram of a piecewise linear map]
    Let $\Phi: |\Sigma| \to \Trop(L)$ be an integral piecewise-linear map as above, then the \emph{diagram} $D_\Phi$ is the $n\times m$ integral matrix whose rows are the images $\Phi(\bv_\rho)$ of the ray generators of the rays $\rho \in \Sigma(1)$. 
\end{definition}

\begin{corollary}\label{cor-diagram}
    Let $\Sigma$ be a smooth fan in $N_\R \cong\R^d$, and let $\Phi: |\Sigma| \to \Trop(L)$ be an integral piecewise-linear map, then $\Phi$ is determined by the diagram $D_\Phi$. Moreover, if $D$ is an integral $n\times m$ matrix with rows in $\Trop(L)$ satisfying the property that for any $\sigma \in \Sigma$, the rows corresponding to the elements in $\sigma(1)$ all lie in a common apartment, then the data $(L,D)$ determines a toric vector bundle over $X_\Sigma$.
\end{corollary}

\begin{proof}
    Let $B$ be a basis such that $\Phi_L(|\sigma|) \subset A_B \subset \Trop(L)$.  The linearity of $\pi_B\circ\Phi_L\!\!\mid_\sigma: |\sigma| \to \R^r$ implies that the image $\Phi_L(p)$ for any $p \in |\sigma|$ can be computed from the $\Phi_L(\bv_\rho)$ for $\rho \in \sigma(1)$. This implies that if $D_\Phi = D_{\Phi'}$ for two integral piecewise-linear maps $\Phi, \Phi': |\Sigma| \to \Trop(L)$, then we must have $\Phi = \Phi'$.

    Now suppose that $D$ is an $n\times m$ integral matrix with the property that the rows $w_\rho$ corresponding to the rays $\rho \in \sigma(1)$ all lie in a common apartment $A_B \subset \Trop(L)$.  For any $p \in |\sigma|$, we write $p = \sum r_\rho \bv_\rho$, and define $\Phi_D(p) = \phi_B(\sum r_\rho \pi_B(w_\rho))$.  By construction, this map is integral and piecewise-linear with $\Phi_D(|\sigma|) \subset A_B$.
\end{proof}

\begin{remark}
    The restriction that $\Sigma$ be a simplicial fan in Corollary \ref{cor-diagram} is minor.  Corollary \ref{cor-diagram} can be extended to any fan $\Sigma$ if we also require that the rows $w_\rho$ corresponding to $\rho \in \sigma(1)$ satisfy any linear relations which hold among the ray generators $u_\rho$.
\end{remark}

\begin{remark}
    Corollary \ref{cor-diagram} implies that the data $(L, D)$ determines a toric vector bundle over $X_\Sigma$, however a given toric vector bundle can be defined by many such pairs.
\end{remark}

\subsection{Toric vector bundles as tropical points}
\label{subsec-tvb-PL-valuation}
We start by extending the notion of a valuation by allowing the value set to be an idempotent semifield (see \cite{GG}).
Let $\mathcal{O}$ be an idempotent semifield, i.e. that is, $\mathcal{O}$ is equipped with binary operations $\oplus$ and $\otimes$ that satisfy the axioms of a field except that $\oplus$ does not have additive inverses. Instead, for any $a \in \mathcal{O}$ we have $a \oplus a = a$. The idempotent operation defines a partial order on $\mathcal{O}$ as follows: for $a, b \in \mathcal{O}$, we say that $a \leq b$ if $a \oplus b = a$. We let $\infty$ denote the neutral element with respect to $\oplus$. 


\begin{definition}[Vector space valuation]
As before let $E \cong \k^r$ be an $r$-dimensional $\k$-vector space. A map $\vv: E \to \mathcal{O}$ is a \emph{valuation} if:
\begin{itemize}
\item[(a)] $\vv(f+g) \geq \vv(f) \oplus \vv(g)$, for all $f, g \in E$,
\item[(b)] $\vv(Cf)=\vv(f)$, for any $0 \neq C \in \k$ and $f \in E$, 
\item[(c)] $\vv(f)=\infty$ if and only if $f=0$.
\end{itemize}
We say that $\vv$ is a \emph{finite} valuation if $\vv(E)$ is a finite set (we note that unlike the case of valuations with values in a totally ordered set, it is possible to have a valuation on a finite dimensional vector space with an infinite set of values).
\end{definition}

One can also define the notion of a valuation on an algebra with values in an idempotent semifield. 

A classic example of an idempotent  semifield is the set $\overline{\R} = \R \cup \{\infty\}$ with the operations of addition for $\otimes$ and taking minimum for $\oplus$. The semifield $(\overline{\R}, \min, +)$ is usually referred to as the \emph{tropical semifield}. The sets $\overline{\Z}$ and $\overline{\Q}$ are subsemifields. 

Next important example of an idempotent semifield is the semifield of piecewise linear functions. As usual let $N \cong \Z^n$ be a free rank $n$ lattice with $N_\R $. Recall that a function $\phi: N_\R \to \R$ is \emph{piecewise linear} if there exists a complete fan $\Sigma$ in $N_\R$ such that $\phi$ is linear restricted to each cone of $\Sigma$. We denote the set of all piecewise linear functions on $N_\R$ by $\PL(N_\R, \R)$. Moreover, we add a unique ``infinity element'' $\infty$ to $\PL(N_\R, \R)$ which is greater than any other element. It is straightforward to see that $\PL(N_\R, \R)$ together with operations of taking minimum $\min$ and addition of functions $+$ is an idempotent semifield. 
One sees that for $\phi_1, \phi_2 \in \PL(N_\R, \R)$ we have $\phi_1 \leq \phi_2$, that is, $\phi_1 \oplus \phi_2 = \phi_1$, if and only if $\phi_1(x) \leq \phi_2(x)$ for all $x \in N_\R$.
 
We also denote the set of piecewise linear functions that attain integer values on $N$ by $\PL(N,\Z)$. Finally, for a complete fan $\Sigma$, we denote by $\PL(\Sigma, \R)$ the set of piecewise linear functions that are linear on cones in $\Sigma$ and $\PL(\Sigma,\Z)$ the subset of piecewise linear functions that have integer values on $N$.

In \cite{Kaveh-Manon-TVBs-valuation}, a finite valuation with values in $\PL(N, \Z)$ is called a \emph{piecewise linear valuation}.

Let $\Sigma$ be a complete fan. A piecewise linear map $\Phi: |\Sigma| \to \tilde{\B}(E)$ gives a piecewise linear valuation $\vv = \vv_\Phi: E \to \PL(N, \Z)$ as follows: 
$$\vv(e)(x) = \Phi(x)(e), \quad \forall x \in |\Sigma| = N_\R.$$ 
Conversely, one can show that for any piecewise linear valuation $\vv$ on $E$, there exists a piecewise linear map $\Phi$ such that $\vv = \vv_\Phi$. The map $\Phi$ is unique up to refining the fan $\Sigma$. The following is proved in \cite{Kaveh-Manon-TVBs-valuation}:
\begin{theorem}[Toric vector bundles as piecewise linear valuations]  \label{th-tvb-PL-valuation}
The equivalence classes of toric vector bundles over $T$-toric varieties up to pull-back via toric morphisms, are in one-to-one correspondence with the set of piecewise linear valuations $\vv: E \to \PL(N, \Z)$, where as before $E$ is the fiber over the distinguished point $x_0$ in the open $T$-orbit. 
\end{theorem}

Finally, in \cite[Section 4]{Kaveh-Manon-TVBs-valuation} the data of a piecewise linear valuation on $E$ is interpreted as a tropical point on a linear ideal over the semifield $\PL(N, \Z)$. Let $\M = \{e_1, \ldots, e_m\} \subset E$ be a finite spanning set. We regard $\M$ as (the ground set of) a linear matroid in the vector space $E$.
Let $L \subset \k[x_1, \ldots, x_m]$ be the linear ideal of relations among the $e_i$. Given $(\phi_1, \ldots, \phi_m) \in \PL(N, \Z)^m$, one can ask when there is a piecewise linear valuation $\vv: E \to \PL(N, \Z)$ with $\vv(e_i) = \phi_i$, for all $i$. The following theorem answers this (see \cite{Kaveh-Manon-TVBs-valuation}):

\begin{theorem}
Let $(\phi_1, \ldots, \phi_m) \in \PL(N, \Z)^m$. The following are equivalent:
\begin{itemize}
\item[(a)] There exists a piecewise linear valuation $\vv: E \to \PL(N, \Z)$ with $\vv(b_i) = \phi_i$ for all $i$ (one shows that $\vv$ is unique, whenever it exists). 
\item[(b)] $(\phi_1, \ldots, \phi_m) \in \Trop_{\PL(N, \Z)}(L)$. 
\item[(c)] For any circuit $C$ in the matroid $\M$ and any $x \in N_\R$, the minimum $\min\{ \phi_i(x) \mid i \in C\}$ is attained twice (see Theorem \ref{th-GF-trop-linear}(c)).  
\end{itemize}
\end{theorem}

\begin{corollary}[Toric vector bundles as tropical points]
With notation as above, the points in $\Trop_{\PL(N,\Z)}(L)$ correspond to toric vector bundles (up to pull-back by toric blowups). Moreover, every toric vector bundle arises in this way.   
\end{corollary}

\section{Preliminaries on matroids and Bergman fans}
\label{sec-prelim-matroids}

Throughout $\M$ denotes a (not necessarily representable) loop-free matroid with ground set $\{1, \ldots, m\}$. By abuse of notation we denote the ground set also by $\M$. We denote the rank of $\M$ by $r$.

Recall that $\{\delta_1, \ldots, \delta_m\}$ denotes the standard basis for $\R^\M$ and for a subset $J \subset \{1, \ldots, m\}$ we put $\delta_J = \sum_{j \in J} \delta_j$. 
Motivated by the case of linear matroids one defines the following:
\begin{definition}   \label{def-GF-polytope-matroid}
Recall that the \emph{matroid polytope} $P_\M$ is the convex hull of $\{\delta_B \mid B \subset \M \text{ is a basis}\}$. The \emph{Gr\"obner fan} $\GF(\M)$ is the outer normal fan of the matroid polytope $P_\M$.
\end{definition}



One shows that every $\delta_B$ is a vertex of $P_\M$. By definition of normal fan, the cones in the Gr\"obner fan $\GF(\M)$ are in one-to-one correspondence with the faces of $P_\M$. In particular, maximal cones in $\GF(\M)$ correspond to bases of $\M$.  Let $\sigma_F$ denote the face of $\GF(\M)$ corresponding to a face $F$ of the matroid polytope $P_\M$.  One shows that the bases of $\M$ corresponding to the vertices of $F$ define a matroid $\M_F$ (on the ground set $\M$) called the \emph{initial matroid} of $\M$ associated to $F$. We let $\In_F(\M)$ denote the initial matroid associated to a face $F$. 


For $e \in \M$ let $\pi_e: \R^\M \to \R$ be projection on the $e$-th coordinate. For a circuit $C \subset \M$, let $\pi_C = \min\{\pi_i \mid i \in C\}$.  Let $\sigma \in \GF(\M)$ be a face, and take $w \in \sigma^\circ$, the relative interior of $\sigma$. Then for any circuit $C$, there are \emph{winner} coordinates in $w$, that is, $i \in C$ such that $w_i = \pi_C(w) = \min\{w_j \mid j \in C\}$.  The collections of winners, for all possible circuits $C$, uniquely determines a cone $\sigma$ in $\GF(\M)$.

Next we recall the Bergman fan of a matroid $\M$ which is a generalization of the tropical variety of a linear ideal (cf. Theorem \ref{th-GF-trop-linear}). 
\begin{definition}[Bergman fan]  \label{def-Bergman-fan}
Let $\F=(F_1 \subsetneqq \cdots \subsetneqq F_k = \M)$ be a flag of flats of $\M$. We define the convex polyhedral cone $\sigma_\F$ by: $$\sigma_\F = \operatorname{cone}\{ \e_F \mid F \in \F \}.$$
The cone $\sigma_\F$ can be described as the set of all points $w \in \R^\M$ satisfying the following conditions: the coordinates $w_i$, $i \in F_1$, are all equal to each other. The coordinates $w_i$, $i \in F_2 \setminus F_1$ are equal to each other and greater than or equal to those in $F_1$. The coordinates $w_i$, $i \in F_3 \setminus F_2$, are all equal to each other and greater than or equal to those in $F_2$ and so on. 

The \emph{Bergman fan} $\Berg(\M)$ is the (usually non-complete) fan in $\R^\M$ consisting of the cones $\sigma_\F$ for all the flags of flats in $\M$. 
\end{definition}

The following is well-known (see \cite[Proposition 2.5]{Feichtner-Sturmfels}):
\begin{proposition} 
The support of $\Berg(\M)$ is the support of a subfan of the Gr\"obner fan $\GF(\M)$. The Bergman fan consists of cones $\sigma \in \GF(\M)$ such that  the initial matroid for corresponding faces $F$ in $P_\M$ are loop-free.
\end{proposition}

\begin{lemma}   \label{lem-filt-Bergman-fan}
Let $w \in \Berg(\M)$. Then for any $r \in \R$, the set:

$$F^w_r = \{ i \in \M \mid w_i \geq r\},$$
is a flat in $\M$. Thus, $(F^w_r)_{r \in \R}$ is a decreasing $\R$-filtration by flats, and $w \in \Berg(\M)$ is determined by this filtration. Conversely, a decreasing $\R$-filtration by flats of $\M$ determines a point $w \in \Berg(\M)$ by 

\begin{equation}  \label{equ-w-F-k}
w_i = \sup\{r \in \R \mid i \in F_r\}.
\end{equation}
Here we assume that for $r$ sufficiently small, $F_r = \M$ and for $k$ sufficiently large, $F_r = \emptyset$.
\end{lemma}

\begin{proof}
Let $j \in \M$ be in the span of $F^w_r$. Then there is a circuit $C$ such that $j \in C$ and $C \setminus \{j\} \subset F^w_r$. We would like to show $j \in F^w_r$. If not, then $w_j < r$. But $w_\ell \geq r$, for all other $\ell \in C$. This contradicts the fact that $\min\{w_i \mid i \in C\}$ is attained twice. To prove the converse, we need to show that given a decreasing filtration $(F_r)_{r \in \R}$ by flats of $\M$, the corresponding $w \in \R^\M$, defined by \eqref{equ-w-F-k}, lies in $\Berg(\M)$. To this end, let $C$ be a circuit and suppose by contradiction that $\min\{w_i \mid i \in C\}$ is attained once at $j \in C$. Then we can find $r \in \R$ such that $j \notin F_r$ but $C \setminus \{j\} \subset F_r$. This contradicts that $F_r$ is a flat.
\end{proof}

\begin{remark}
In the case where $\M$ is a representable matroid corresponding to a linear ideal $L$, each $w \in \Trop(L)$ corresponds to a valuation $v: E \to \overline{\R}$. The valuation $v$ on $E$ is determined by the decreasing $\R$-filtration of vector subspaces $(E_{v \geq r})_{r \in \R}$ and this filtration uniquely determines $v$ and hence $w$ (Section \ref{subsec-tvb-PL-map-trop}). The above filtration $(F^w_r)_{r \in \R}$ is an extension of this situation to all matroids. 
\end{remark}

In analogy with the theory of buildings and classification of toric vector bundles in \cite{Kaveh-Manon-TVBs-valuation, Kaveh-Manon-Building} in terms of piecewise linear maps to buildings, we introduce the notion of an ``apartment'' in the Bergman fan. Each apartment is a subset of the Bergman fan obtained by intersecting it with a maximal face of the Gr\"obner fan. Below we show that each apartment is piecewise linearly isomorphic (hence homeomorphic) to a real vector space of dimension equal to $\rank(\M)$. 
\begin{definition}[Apartment in Bergman fan] \label{def-apartment-Bergman-fan}
Let $B \subset \M$ be a basis, and let $\sigma_B$ be the corresponding maximal cone in the Gr\"obner fan $\GF(\M)$. We define the \emph{apartment} $A_B$ to be the intersection $$A_B = \Berg(\M) \cap \sigma_B.$$
\end{definition}

\begin{proposition}\label{prop-apartment-matroid}
Let $B$ be a basis in $\M$ with corresponding apartment $A_B$. We have the following:
\begin{itemize}
\item[(a)] $A_B$ is a union of cones in the Bergman fan (and hence has structure of a simplicial complex where each simplex is a cone). \item[(b)] $A_B$ is piecewise linearly isomorphic to $\R^r$ where $r = \rank(\M)$.
\item[(c)] As a simplicial complex, $A_B$ is isomorphic to the Coxeter complex of type $A_{r-1}$ where $r=\rank(\M)$. 
\end{itemize}
\end{proposition}
\begin{proof}
Let $B \subset \M$ be a basis. For a total ordering $\prec$ on the set $B$, let $\sigma_\prec \subset \R^B$ be the subset of those $w \in \R^B$ such that $w_b \leq w_{b'}$ whenever $b \prec b'$. The set $\sigma_\prec$ is a closed cosimplicial cone and the collection of the $\sigma_\prec$ and their faces, for all orderings $\prec$, form the permutahedral fan in $\R^B$.
As in Section \ref{subsec-tvb-PL-map-trop}, we define a piecewise-linear map $\phi_B: \R^B \to \Berg(\M)$. For $i \in \M \setminus B$, let $C_i$ be the unique circuit containing $i$ and such that $C_i \setminus \{i\} \subset B$. Now, for $a \in \R^B$ we define $\phi_B(a) = w \in \R^\M$ where:
\begin{equation}  \label{equ-phi-B}
w_i = 
\begin{cases}
a_i \quad \textup{for } i \in B\\
\min\{ a_j \mid e_j \in C_i \setminus \{e_i\} \} \quad \textup{for } e_i \notin B.
\end{cases}
\end{equation}
We now show that, for any total ordering on $B$, $\phi_B$ maps $\sigma_\prec$ linearly and bijectively onto a cone $\sigma_\F$ in $\Berg(\M)$ for some flag of flats $\F$ (see Definition \ref{def-Bergman-fan}). 
Let us order elements of the basis $B$ as $b_1 \prec b_2 \prec \cdots \prec b_r$ and define a flag of flats $\F = (F_1 \subsetneqq \cdots \subsetneqq F_r = \M)$ by taking $F_i$ to be the span of $\{b_1, \ldots, b_i\}$. Then one verifies that for for $a \in C_\prec$, the point $\phi_B(a)$ lies in the cone $\sigma_\F$. Conversely, 
for any $w \in \sigma_\F$, we have $w = \phi_B(a)$ where $a=\pi_B(w)$. Recall that $\pi_B: \R^\M \to \R^B$ is the projection onto coordinates in $B$. This finishes the proof.
\end{proof}

\section{Tropical toric vector bundles}\label{sec-matroid-vb}
In this section we introduce the main concept of the paper, namely a \emph{tropical toric vector bundle}. We give two equivalent definitions for this concept inspired by the classification of toric vector bundles (see Sections \ref{subsec-tvb-PL-map-trop} and \ref{subsec-tvb-PL-valuation}). 

\subsection{Two equivalent definitions of a tropical toric vector bundle} \label{subsec-def-matroid-vb}
Our first definition of a tropical toric vector bundle is an extension of the description of a toric vector bundle as a piecewise linear map to a tropical linear ideal (Section \ref{subsec-tvb-PL-map-trop}).

Let $\Sigma$ be a fan in $N_\R$ and let $\M$ be a matroid. For simplicity we assume $\Sigma$ is a complete fan.
\begin{definition}[Tropical toric vector bundle as a piecewise linear map] \label{def-matroid-vb-1}
A \emph{tropical toric vector bundle} over the toric variety $X_\Sigma$ is the data of a  map $\Phi: |\Sigma| \to \Berg(\M)$ with the following property: for any cone $\sigma \in \Sigma$ there is a (not necessarily unique) basis $B_\sigma$ in $\M$ such that $\Phi(\sigma)$ lands in the apartment $A_{B_\sigma}$ (Definition \ref{def-apartment-Bergman-fan}), and moreover $\Phi\!\!\mid_{\sigma}: \sigma \to A_{B_\sigma}$ is the restriction of an $\R$-linear map $\Phi_\sigma: \Span{\sigma} \to A_{B_\sigma}$, where we identify the apartment $A_{B_\sigma}$ with the vector space $\R^r$ as in Proposition \ref{prop-apartment-matroid}(b).  
\end{definition}

Next, we present an (equivalent) definition of a tropical toric vector bundle which is an extension of the description of a toric vector bundle as a $\PL(N, \Z)$-valued point on a tropical linear ideal (Section \ref{subsec-tvb-PL-valuation}).
\begin{definition}[Tropical toric vector bundle as a tropical point]  \label{def-matroid-vb-2}
A \emph{tropical toric vector bundle} is the data of a map $\vv: \M \to \PL(N, \Z)$ such that:
\begin{enumerate}
\item for any $x \in N$, any any circuit $C$, the minimum of the values $\vv(e)(x)$ for $e \in C$ is attaned twice,  
\item For any cone $\sigma \in \Sigma$ there is a basis $B_\sigma$ such that for any $x \in \sigma$ and any circuit $C = \{e, I\}$ with $I \subset B_\sigma$ and $e \notin B_\sigma$ we have:
$$\vv(e)(x) = \min\{\vv(b)(x) \mid b \in I \}.$$
\end{enumerate}

\end{definition}

One immediate consequence of Definition \ref{def-matroid-vb-2} is that for any $b \in B_\sigma$, the restriction $\vv(b)\!\!\mid_{\sigma}$ is an integral linear function on $\sigma$, namely $\vv(b)\!\!\mid_{\sigma} \in M_\sigma := M / (M \cap \sigma^\perp)$. Thus, when $\sigma$ is full dimensional we consider $\vv(b)\!\!\mid_{\sigma} \in M$.

\begin{remark}
We would have liked to call $\vv:\M \to \PL(N, \Z)$ a \emph{matroid valuation} with values in $\PL(N, \Z)$, but this terminology is already taken and valuation on a matroid means something else in the literature. 
\end{remark}

\begin{proposition}
Definitions \ref{def-matroid-vb-1} and \ref{def-matroid-vb-2} are equivalent.   
\end{proposition}
\begin{proof}
This is a straightforward consequence of definitions. First, let $\Phi: |\Sigma| \to \Berg(\M)$ satisfy the condition of Definition \ref{def-matroid-vb-1}. We build a function $\vv: \M \to \PL(N, \Z)$ as follows: for any $e \in \M$ and any ray $\rho$ with corresponding primitive vector $\bv_\rho$, $\vv(e)$ is the piecewise linear function whose value on $\bv_\rho$ is given by:
$$\vv(e)(\bv_\rho) = \pi_e(\Phi(\bv_\rho)).$$
Recall that $\pi_e: \Berg(\M) \to \R$ is the projection on the $e$-th coordinate.  By definition, $\vv(e) \in \PL(N, \Z)$.  Let $C \subset \M$ be a circuit. For any $x \in |\Sigma|$ we have $\Phi(x) \in \Berg(\M)$, so we must have that the minimum of $\{\Phi(x)(c) \mid c \in C\}$ occurs twice. It follows that the function $\min\{\vv(c) \mid c \in C\} \in \PL(N, \Z)$ is unchanged by the removal of any element $c \in C$.  This means that $\vv: \M \to \PL(N, \Z)$ is a $\PL(N, \Z)$-valued point on $\Berg(\M)$. Now for any cone $\sigma \in \Sigma$ let $\B_\sigma$ be a basis as in Definition \ref{def-matroid-vb-2}. Then for any circuit $\{e, I\}$ with $I \subset B_\sigma$, and $x \in \sigma$, we must have $\Phi(\bv_\rho) \in A_{B_\sigma}$. By Definition \ref{def-apartment-Bergman-fan}, this implies that $\pi_e(\Phi(x)) = \min\{\pi_b(\Phi(x)) \mid b \in I\}$.  As a consequence, we must have $\vv(e)(\bv_\rho) = \min\{v(b)(x) \mid b \in I\}$. 

Conversely, we may run these arguments in reverse. Supposing that $\vv:\M \to \Berg(\M)$ is a $\PL(N, \Z)$-valued point on $\Berg(\M)$. We define $\Phi(x) \in \R^\M$ to be the tuple obtained by evaluating $\vv(e)(x)$ for all $e \in \M$. By Definition \ref{def-matroid-vb-2}, $\Phi(x) \in \Berg(\M)$. Also, for any cone $\sigma \in \Sigma$, the basis $B_\sigma \subset \M$ serves to define the apartment $A_{B_\sigma} \subset \Berg(\M)$, and the condition $\vv(e)(x) = \min\{v(b)(x) \mid b \in I\}$ implies that $\Phi(x) \in A_{B_\sigma}$ for any $x \in \sigma$. 
\end{proof}

We will refer to either of the data in Definitions \ref{def-matroid-vb-1} or \ref{def-matroid-vb-2} as a \emph{tropical toric vector bundle} $\EE$ with \emph{piecewise linear map} $\Phi$ and \emph{$\PL(N, \Z)$-valued tropical point} $\vv$. We also address $\EE$ by $(\M, \Phi)$ or $(\M, \vv)$.  

\begin{definition}[Klyachko flats and diagram of a tropical toric vector bundle]  \label{def-Klyachko-flats}
For each ray $\rho \in \Sigma(1)$, let $w_\rho := \Phi(\bv_\rho) \in \Berg(\M)$. To $w_\rho$ there corresponds a $\Z$-filtration $\cdots F^{\rho}_{r} \supseteq F^\rho_{r+1} \cdots$ by flats of $\M$. This is the analogue of Klyachko filtrations for toric vector bundles (Section \ref{subsec-prelim-toric-vb}). 

Let $n = |\Sigma(1)|$ be the number of rays of $\Sigma$, then for $r_1, \ldots, r_n \in \Z$, the corresponding Klyachko flat is the intersection $F^{\rho_1}_{r_1}\cap \cdots \cap F^{\rho_n}_{r_n} \subseteq \M$.

We let $D_\Phi$ be the $n \times m$ integral matrix with rows the $w_\rho$ for $\rho \in \Sigma(1)$, where $m = |\M|$. We call the matrix $D_\Phi$ the \emph{diagram} of the tropical toric vector bundle determined by $\Phi$.
\end{definition}

\begin{proposition}
    If $\Phi$, $\Phi': |\Sigma| \to \Berg(\M)$ are piecewise-linear maps with $D_\Phi = D_{\Phi'}$, then $\Phi = \Phi'$. Let $\Sigma$ be a simplicial fan. If $D$ is any $n\times m$ integral matrix with rows in $\Berg(\M)$ such that any rows corresponding to rays of a face $\sigma \in \Sigma$ lie in a common apartment of $\Berg(\M)$, then $D$ determines a piecewise-linear map $\Phi_D: |\Sigma| \to \Berg(\M)$. 
\end{proposition}

\begin{proof}
    The proof is the same as the proof of Corollary \ref{cor-diagram}.
\end{proof}

\begin{definition}[Parliament of polytopes] \label{def-parliament}
    Let $\Phi: |\Sigma| \to \Berg(\M)$ define a tropical toric vector bundle $\EE$ with diagram $D$. 
    For each $e \in \M$ we let $P_{\vv(e)} \subset M_\R$ be the Newton polytope of the divisor on $X_\Sigma$ defined by the $e$-th column of $D$:  
    \begin{equation} \label{equ-P-v(e)}
P_{\vv(e)} = \{ y \in M_\R \mid \langle y, \bv_\rho \rangle \leq \vv(e)(\bv_\rho),~\forall \rho \in \Sigma(1) \}.
\end{equation}
The \emph{parliament} of $\EE$ is defined to be the collection of polyhedra $\{P_{\vv(\e)} \mid e \in \M\}$.
\end{definition}

{\begin{definition}[Tropicalization of a toric vector bundle]\label{def-tropicalization}
    Let $X_\Sigma$ be a toric variety, let $\E$ be a toric vector bundle over $X_\Sigma$, and suppose $\E$ is determined by the data $(L, \Phi)$ as in Definition \ref{def-PL-trop-lin}. The \emph{tropicalization} of $\E$ is the pair $(\M(L), \Phi)$. 
\end{definition}

As is usually the case in tropical geometry, the tropicalization depends on the choice $(L, \Phi)$ of presentation data for $\E$.  For the following notion see \cite{DJS}.

\begin{definition}[DJS matroid of a toric vector bundle]\label{def-DJSmatroid-toricvectorbundle}
    Let $\E$ be a toric vector bundle determined by the data $(L, \Phi)$, we say $\M(L)$ is a \emph{DJS matroid} if the ranks of the Klyachko flat $F^{\rho_1}_{r_1}\cap \cdots \cap F^{\rho_n}_{r_n} \subseteq \M(L)$ always coincides with the dimension of the Klyachko space $E^{\rho_1}_{r_1} \cap \cdots \cap E^{\rho_n}_{r_n}$, and $\M(L)$ is minimal with respect to this property.  If $\M(L)$ has this property, but is not minimal, we say that $\M(L)$ is an extension of a DJS matroid (see Section \ref{sec-matroid-extension}). 
\end{definition}
}


\begin{example}\label{ex-Fano}
    We describe a tropical toric vector bundle $\FF$ over $\P^2$ built from the Fano plane $\mathcal{F}$. (Figure \ref{fig-Fano}).

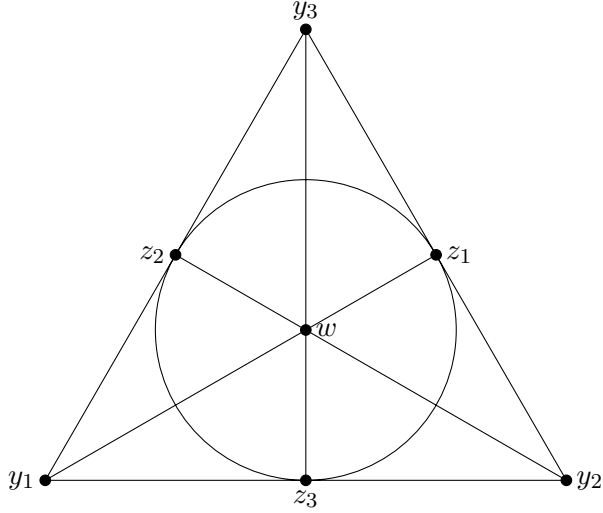
\begin{figure}[ht]
\begin{tikzpicture}
\begin{scope}

\draw (30:2)  -- (210:4)
        (150:2) -- (330:4)
        (270:2) -- (90:4)
        (90:4)  -- (210:4) -- (330:4) -- cycle
        (0:0)   circle (2);

\filldraw[black] (0:0) circle (2pt) node[anchor=west]{$w$};

\filldraw[black] (30:2) circle (2pt) node[anchor=west]{$z_1$};
\filldraw[black] (150:2) circle (2pt) node[anchor=east]{$z_2$};
\filldraw[black] (270:2) circle (2pt) node[anchor=north]{$z_3$};

\filldraw[black] (90:4) circle (2pt) node[anchor=south]{$y_3$};
\filldraw[black] (210:4) circle (2pt) node[anchor=east]{$y_1$};
\filldraw[black] (330:4) circle (2pt) node[anchor=west]{$y_2$};

\end{scope}
\end{tikzpicture}
\caption{The Fano plane.} \label{fig-Fano}
\end{figure}
   
    Let $\rho_1, \rho_2, \rho_3$ be the rays of the fan of $\P^2$. To describe a matroid bundle over $\P^2$ it suffices to find three flags of flats in the Fano plane such that any pair of flags shares a common adapted basis.  We let $F^{\rho_i}_r = \mathcal{F}$ for $r \leq 0$, $F^{\rho_i}_r = \{w,y_i,z_i\}$ for $0 < r \leq 1$, $F^{\rho_i}_r = \{y_i\}$ for $1 < r \leq 2$, and $F^{\rho_i}_r = \emptyset$ for $2 < r$.  The flags for rays $\rho_i$ and $\rho_j$ share the basis $\BB_{ij} = \{y_i, y_j, w\}$.  This information is encoded in the the following diagram:

\begin{center}
\begin{tabular}{ c|ccccccc } 
  & $y_1$ & $y_2$ & $y_3$ & $z_1$ & $z_2$ & $z_3$ & $w$\\ 
 \hline
 $\rho_1$ & 2 & 0 & 0 & 1 & 0 & 0 & 1\\ 
 $\rho_2$ & 0 & 2 & 0 & 0 & 1 & 0 & 1\\ 
 $\rho_3$ & 0 & 0 & 2 & 0 & 0 & 1 & 1\\ 
\end{tabular}
\end{center}
\bigskip
    The parliament of polytopes for $\FF$ is composed of the moment polyhedra for 1 divisor, 3 degree 2 divisors, and 3 degree 1 divisors on $\P^2$.  The toric divisors for these polyhedra can be read off the columns of the diagram.

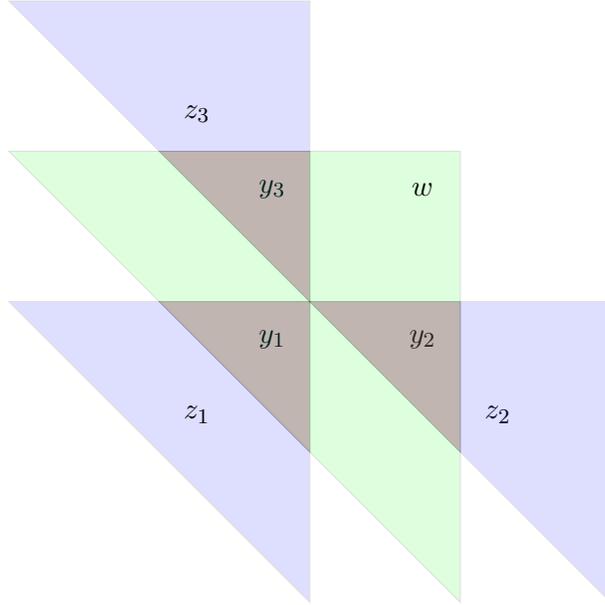
\begin{figure}[ht]
\begin{tikzpicture}
\begin{scope}
\draw[fill=red,opacity=0.23] (-2,0) -- (0,0) -- (0,-2) -- cycle;
\draw (-.5,-.5) node{$y_1$};
\draw[fill=red,opacity=0.23] (0,2) -- (0,0) -- (-2,2) -- cycle;
\draw (1.5,-.5) node{$y_2$};
\draw[fill=red,opacity=0.23] (0,0) -- (2,0) -- (2,-2) -- cycle;
\draw (-.5,1.5) node{$y_3$};
\draw[fill=blue,opacity=0.13] (-4,0) -- (0,0) -- (0,-4) -- cycle;
\draw (-1.5,-1.5) node{$z_1$};
\draw[fill=blue,opacity=0.13] (0,4) -- (0,0) -- (-4,4) -- cycle;
\draw (2.5,-1.5) node{$z_2$};
\draw[fill=blue,opacity=0.13] (0,0) -- (4,0) -- (4,-4) -- cycle;
\draw (-1.5,2.5) node{$z_3$};
\draw[fill=green,opacity=0.13] (-4,2) -- (2,2) -- (2,-4) -- cycle;
\draw (1.5,1.5) node{$w$};
\end{scope}
\end{tikzpicture}
\caption{The parliament of polytopes for the bundle $\FF$.} \label{fig-PARLIAMENT}
\end{figure}
The integral point $(1, 0)$ lies in the polytopes associated to $w, y_2, z_2$. As a consequence, we see that the global section matroid $H^0(\P^2,\FF)_{(1,0)}$ is the rank $2$ flat $\{w, y_2, z_2\} \subset \mathcal{F}$. 
\end{example}

\section{Equivariant $K$-class and characteristic classes of tropical toric vector bundles}
\label{sec-K-class-char-class}
Throughout this section, $\Sigma$ is a complete fan. In this section, generalizing the equivariant $K$-class and equivariant Chern classes of toric vector bundles (see \cite[Proposition 3.1]{Payne-moduli}, \cite[Corollary 3.5]{Kaveh-Manon-Building}), we associate a piecewise exponential and piecewise polynomial functions to a tropical toric vector bundle $\EE$. We interpret them as \emph{equivariant $K$-class} and \emph{equviariant Chern classes} of the tropical toric vector bundle $\EE$.

In construction of both equivariant classes, a crucial step is to construct universal real-valued continuous functions on $\Berg(\M)$ such that the required equivariant classes are obtained by composing the piecewise linear map $\Phi: |\Sigma| \to \Berg(\M)$ with these universal functions. The next lemma is the key to this construction.

\begin{lemma}\label{lem-universal-function}
For any continuous function $g: \R^r \to \R$ that is invariant under the action of symmetric group on $r$ letters there is a unique (continuous) function $g_\M: \Berg(\M) \to \R$, with the property that $\phi_B^*g_\M = g$ for any basis $B$. Here $\phi_B$ is the piecewise linear bijection as in \eqref{equ-phi-B}.
\end{lemma}
\begin{proof}
We show that the function $g_\M$ whose restriction to each apartment coincides with $g$ is well-defined. By the exchange property, it is enough to show that $g_\M$ is well-defined and continuous on the union of two apartments $A_1, A_2$ whose corresponding bases differ $B_1, B_2$ differ by a single element. We let $B_1 = \{b_1, \ldots, b_{r-1}, x\}, B_2 = \{b_1, \ldots, b_{r-1}, y\}$, then there is a circuit $C = \{x, y\} \cup I$, where $I \subset \{b_1, \ldots, b_{r-1}\}$. Let $a = (a_1, \ldots, a_m) \in A_1 \cap A_2$.  The basis $B_1$ identifies $A_1$ with $\R^r$ by sending $a$ to the entries $a_1, \ldots, a_{r-1}, a_x$ corresponding to $b_1, \ldots, b_{r-1}$ and $x$ (see Equation \eqref{equ-phi-B}). Likewise, the basis for $B_2$ is used to identify $a$ with the point in $\R^r$ whose coordinates are its $a_1, \ldots, a_{r-1}, a_y$ entries. The circuit $C$ tells us that $a_x = \min\{a_y, a_i \mid b_i \in I \} \leq a_y$.  Likewise, $a_y \leq a_x$. We thus conclude that $a_x = a_y$, and hence ${g_\M}_{|A_1 \cup A_2}$ is well-defined and continuous. 
\end{proof}

\subsection{Equivariant $K$-class of a tropical toric vector bundle}
\label{subsec-equiv-K-class}
By \cite{Anderson-Payne}, a class $[\E] \in K^0_T(X_\Sigma)$ corresponding to a tropical toric vector bundle $\E$ on $X_\Sigma$ should be a piecewise exponential function on the support $|\Sigma|$. By the localization theorem, there is an injection $i: K^0_T(X_\Sigma) \hookrightarrow \prod_{x_\sigma \in X_\Sigma} \Z[T]$, where the product is over the torus fixed points $x_\sigma \in X_\Sigma$ corresponding to the maximal cones $\sigma \in \Sigma(n)$. Each copy of the polynomial ring $\Z[T]$ is viewed as the representation ring of the torus $T$ (which itself can be identified with the polynomial ring $\Z[M]$ on the character lattice $M$ of $T$). The image of the injection $i$ consists of tuples of virtual representations of $T$ that satisfy the following compatibility condition: if two maximal cones $\sigma$ and $\sigma'$ have a codimension $1$ face $\tau$, that is, if the fixed points $x_\sigma$ and $x_{\sigma'}$ are connected by a $T$-invariant curve, then the corresponding virtual representations of $T$ agree on the stabilizer of this $T$-curve.

We recall the computation of the equivariant $K$-class of a toric vector bundle. Let $\E$ be a (usual) toric vector bundle over a toric variety $X_\Sigma$ with the corresponding piecewise linear map $\Phi: |\Sigma| \to \tilde{\B}(E)$ (see Section \ref{subsec-prelim-toric-vb}). The image of the class $[\E]$ under the localization map $i$ is the tuple $(f_\sigma)_{\sigma \in \Sigma(n)}$ where $f_\sigma$ is defined as follows. Let $B_\sigma$ and $u(\sigma)$ be the equivariant trivialization data of $\E_{|U_\sigma}$. Then:
$$f_\sigma = \sum_{u_i \in u(\sigma)} \exp(u_i).$$

Let $\exp: \R^r \to \R$ be the sum of exponentials of the coordinates, that is:
$$\exp(x_1, \ldots, x_r) = \sum_{i=1}^r \exp(x_i).$$
In light of Lemma \ref{lem-universal-function}, for any matroid $\M$, we have a well-defined function $\exp_\M: \Berg(\M) \to \R$ which on each apartment coincides with $\exp$. 

\begin{definition}[Equivariant $K$-class of a tropical toric vector bundle] \label{def-K-class-matroid-vb}
Let $\EE$ be a tropical toric vector bundle given by the data of a piecewise linear map $\Phi: |\Sigma| \to \Berg(\M)$. The equivariant $K$-class $[\EE]$ of $\EE$ is the piecewise exponential function $f: |\Sigma| \to \R$ given by:
$$f = \exp_\M \circ \Phi.$$
\end{definition}

More explicitly, the function $f$ can be described as follows. By definition of a piecewise linear map, we know that for each maximal cone $\sigma \in \Sigma$, there exists a basis $B_\sigma$ of $\M$ and characters $u_{\sigma, 1}, \ldots, u_{\sigma, r}$, such that $\Phi_{|\sigma}: \sigma \to A_\sigma \cong \R^r$ is given by: 
\begin{equation}  \label{equ-Phi-sigma}
\Phi_{|\sigma}(x) = (\langle u_{\sigma, 1}, x \rangle, \ldots, \langle u_{\sigma, r}, x \rangle), \quad \forall x \in \sigma.
\end{equation}
Thus, $f_\sigma = f_{|\sigma}$ is given by:
$$f_\sigma(x) = \sum_{i=1}^r \exp(\langle u_{\sigma, i}, x \rangle), \quad \forall x \in \sigma.$$


\subsection{Equivariant Chern classes of a tropical toric vector bundle}
\label{subsec-equiv-Chern-class}
Similarly to the construction of the equivariant $K$-class, we can construct equivariant Chern classes of a tropical toric vector bundle. They generalize those of toric vector bundles which correspond to the case when the matroid is a linear matroid.

Let us recall the description of equivariant Chern classes of a toric vector bundle from \cite[Proposition 3.1]{Payne-moduli} and \cite[Corollary 3.5]{Kaveh-Manon-Building}. Let $\epsilon_i: \R^r \to \R$ be the $i$-th elementary symmetric function in $r$ variables. As shown in \cite[Section 3]{Kaveh-Manon-Building} there is a well-defined map, denoted by the same letter, $\epsilon_i: \tilde{\B}(E) \to \R$ which coincides with the $i$-th elementray symmmetric function on each apartment (which ia copy of $\R^r$). Then for a toric vector bundle $\E$ with corresponding piecewise linear map $\Phi$, the $i$-the equivariant Chern class of $\E$ is represented by the piecewise polynomial function $\epsilon_i \circ \Phi: |\Sigma| \to \R$. 

As above let $\epsilon_i: \R^r \to \R$ be the $i$-th elementary symmetric function and let $\epsilon_{\M, i}: \Berg(\M) \to \R$ be the function (as in Lemma \ref{lem-universal-function}) that coincides with $\epsilon_i$ on each apartment. 

\begin{definition}[Equivariant Chern classes of a tropical toric vector bundle] \label{def-Chern-class-matroid-vb}
Let $\EE$ be a tropical toric vector bundle given by the data of a piecewise linear map $\Phi: |\Sigma| \to \Berg(\M)$. The $i$-th equivariant Chern class of $\EE$ is the piecewise polynomial function $c^T_i(\EE): |\Sigma| \to \R$ given by:
$$c^T_i(\EE) = \epsilon_{\M, i} \circ \Phi.$$
\end{definition}

 For a maximal cone $\sigma \in \Sigma$, take a basis $B_\sigma$ of $\M$ and characters $u_{\sigma, 1}, \ldots, u_{\sigma, r}$ determining the linear map $\Phi_{|\sigma}: \sigma \to \tilde{A}(B_\sigma) \cong \R^r$, as in \eqref{equ-Phi-sigma}. Then, $c^T_{i, \sigma}(\EE) = {c^T_i}(\EE)_{|\sigma}$ is given by:
$$c^T_{i, \sigma}(\EE)(x) = \epsilon_i(\langle u_{\sigma, 1}, x \rangle, \ldots, \langle u_{\sigma, r}, x \rangle), \quad \forall x \in \sigma.$$

\subsection{Equivariant Chern character of a tropical toric vector bundle}
\label{subsec-equiv-Chern-char}
For a toric vector bundle $\E$, the equivariant Chern character $\ch_T(\E)$ is a piecewise polynomial function computed by expanding the exponentials in each restriction ${f_\E}_{|\sigma}$, $\sigma \in \Sigma$, and taking terms of degree less than or equal to $n$ (see \cite[Section 3.6 ]{Brion-Vergne}). This makes sense for a tropical toric vector bundle $\EE$ (over $X_\Sigma$) as well, as we have also defined $[\EE] \in K_T^0(X_\Sigma)$ to be a piecewise-exponential function on $|\Sigma|$. That is, we define the piecewise polynomial function $\ch_T(\EE)$ by:
$$\ch_T(\EE)_{|\sigma}(x) = \textup{sum of terms of degree } \leq n \textup{ in the expansion of } \sum_{i=1}^r \exp(\langle u_{\sigma, i}, x \rangle), \quad \forall \sigma \in \Sigma.$$
One can also give an efficient expression for the equivariant Chern character in terms of power sum functions on $\Berg(\M)$. We have already introduced the elementary symmetric functions $\epsilon_{\M, i}: \Berg(\M) \to \R$ and the exponential function $\exp_\M: \Berg(\M) \to \R$. For an integer $k \geq 0$, let $p_k: \R^r \to \R$ be the $k$-th power function, that is:
$$p_k(x_1, \ldots, x_r) = \sum_{i=1}^r x_i^k.$$
Then, by Lemma \ref{lem-universal-function}, we have power sum functions $p_{\M, k}: \Berg(\M) \to \R$ which are continuous on $\Berg(\M)$. The following identity of functions on $\Berg(\M)$ is then immediate:$$\exp_\M = \sum_{k=1}^r \frac{1}{k!}p_{\M, k}.$$ 
Rewriting the power sum functions $p_{\M, i}$ in terms of the elementary symmetric functions $\epsilon_{\M, i}$ and then composing with $\Phi_\EE$ provides an expression for the Chern character of $\EE$ in terms of the equivariant Chern classes $c^T_i(\EE)$. 

{We conclude this section by recording the connection between the $K$-class and Chern classes of a toric vector bundle and those of a tropicalization. 

\begin{proposition}\label{prop-tropical-characteristic-comparison}
    Let $\E$ be the toric vector bundle over a toric variety $X_\Sigma$ determined by the data $(L, \Phi)$. The equivariant $K$-class and equivariant Chern classes of $\E$ coincide with those of its tropicalization $(\M(L), \Phi)$. In particular, all tropicalizations of $\E$ have the same equivariant $K$-classes and Chern classes.  
\end{proposition}

\begin{proof}
    The equivariant $K$-class and the Chern classes both depend on the map from $|\Sigma|$ to the building $\tilde{\mathcal{B}}(E)$.  This map can be recovered from $\Phi: |\Sigma| \to \Trop(L)$, and coincides with $\Phi: |\Sigma| \to \Berg(\M(L))$. 
\end{proof}
}

\section{Sheaf of sections and Euler characteristic of a tropical toric vector bundle}
\label{sec-Euler-char-matroid-vb}

In this section we define a sheaf of matroids on the toric open cover of $X_\Sigma$ associated to the data of a tropical toric vector bundle.  This allows us to make sense of global generation for a tropical toric vector bundle. 

\subsection{Review of the toric vector bundle case}
\label{subsec-review-sheaf}
Let $\E$ be a toric vector bundle over a complete smooth toric variety $X_\Sigma$. For a character $u \in M$, we let  $H^0(X_\Sigma, \E)_u$ denote the $u$-weight space in the space of global sections $H^0(X_\Sigma, \E)$. Similarly, for any cone $\sigma \in \Sigma$, $H^0(U_\sigma, \E)_u$ denotes the $u$-weight space in the space of sections of $\E$ on the affine toric chart $U_\sigma$. Let $\chi(X_\Sigma, \E)_u$ denote the Euler characteristic of the sheaf of $u$-weight sections, that is:
$$\chi(X_\Sigma, \E)_u = \sum_{i=0}^r (-1)^i \dim H^i(X_\Sigma, \E)_u.$$
We can compute $\chi(X_\Sigma, \E)_u$ using \v{C}ech cohomology with respect to the open cover $\{U_\sigma \mid \sigma \in \Sigma\}$ as follows:
$$\chi(X_\Sigma, \E)_u = \sum_{\sigma \in \Sigma} (-1)^{\codim(\sigma)} \dim H^0(U_\sigma, \E\!\!\mid_{U_\sigma})_u.$$

Let $\vv: E \to \PL(|\Sigma|, \Z)$ be the finite piecewise linear valuation associated to the toric vector bundle $\E$ (see Section \ref{subsec-tvb-PL-valuation}). We recall from \cite[Section 3.4]{Kaveh-Manon-TVBs-valuation} how to read off the dimension of space of $u$-weight sections from the piecewise linear valuation $\vv$.

Following \cite{DJS, Kaveh-Manon-TVBs-valuation}, these dimensions can be computed in terms of matroid data. To see this, we use the expression for the module $H^0(U_\sigma, \E\!\!\mid_{U_\sigma})$ in terms of the Klyachko spaces of $\E$.  Let $\sigma$ be a maximal cone in $\Sigma$ and let $\sigma(1) = \{\rho_1, \ldots, \rho_d\}$, and let $\bv_i\in N$ be the ray generator of the ray $\rho_i$, then:

$$H^0(U_\sigma, \E\!\!\mid_{U_\sigma}) = \bigoplus_{\br \in \Z^d} E^{\rho_1}_{r_1} \cap \cdots \cap E^{\rho_d}_{r_d}.$$
The polynomial ring $H^0(U_\sigma, \O\!\!\mid_{U_\sigma}) \cong \k[x_1, \ldots, x_d]$ acts on this module by the rule $x_i\cdot (E^{\rho_1}_{r_1} \cap \cdots \cap E^{\rho_d}_{r_d}) \subseteq E^{\rho_1}_{r_1} \cap \cdots\cap E^{\rho_i}_{r_i-1} \cap \cdots \cap E^{\rho_d}_{r_d}$. The graded component $H^0(U_\sigma, \E\!\!\mid_{U_\sigma})_u$ is the space $E^{\rho_1}_{\langle \bv_1, u\rangle} \cap \cdots \cap E^{\rho_d}_{\langle \bv_d, u\rangle}$. The specialization map from $H^0(U_\sigma, \E\!\!\mid_{U_\sigma})$ to the fiber $\E_{\sigma}$ over the torus fixed point of $X_\Sigma$ corresponding to $\sigma$ is the quotient map $H^0(U_\sigma, \E\!\!\mid_{U_\sigma}) \to H^0(U_\sigma, \E\!\!\mid_{U_\sigma})/ \langle x_1, \ldots, x_d\rangle H^0(U_\sigma, \E\!\!\mid_{U_\sigma})$. In terms of Klyachko spaces, this fiber has the expression $\bigoplus_{\br \in \Z^d} E^{\rho_1}_{r_1}\cap \ldots \cap E^{\rho_d}_{r_d}/ \sum_{i = 1}^d E^{\rho_1}_{r_1} \cap \cdots\cap E^{\rho_i}_{r_i+1} \cap \cdots \cap E^{\rho_d}_{r_d}$, and the specialization map is just the quotient map on each graded component. 

For a cone $\sigma \in \Sigma$ and $e \in E$ we can associate 
a polyhedron $P_{\vv(e), \sigma} \subset M_\R$ defined by:
\begin{equation}
P_{\vv(e), \sigma} = \{ y \in M_\R \mid \langle y, \bv_\rho \rangle \leq \vv(e)(\bv_\rho),~\forall \rho \in \sigma(1)\}.
\end{equation}
It follows from the definition that for any character $u \in M$ we have:
\begin{equation}   \label{equ-dim-u-sec-sigma}
\dim H^0(U_\sigma, \E\!\!\mid_{U_\sigma})_u = \dim \{ e \in E \mid u \in P_{\vv(e), \sigma} \}.
\end{equation}

Let $\M \subset E$ be the matroid associated to the subspace arrangement given by:
$$\mathcal{A}_\vv = \{ E_{\vv \geq \phi} \mid \phi \in \PL(N, \Z) \},$$ then 

$$\dim H^0(U_\sigma, \E\!\!\mid_{U_\sigma})_u = \dim (E^{\rho_1}_{\langle \bv_1, u\rangle} \cap \cdots \cap E^{\rho_d}_{\langle \bv_d, u\rangle}) = \rank \{ e \in \M \mid u \in P_{\vv(e), \sigma}\}.$$
This is \cite[Theorem 3.14]{Kaveh-Manon-TVBs-valuation}.

Similarly, for any $e \in E$, we define the polyhedron from the parliament of $\E$:
\begin{equation}
P_{\vv(e)} = \{ y \in M_\R \mid \langle y, \bv_\rho \rangle \leq \vv(e)(\bv_\rho),~\forall \rho \in \Sigma(1) \}.
\end{equation}
We also have $H^0(X_\Sigma, \E)_u = \bigcap_{\rho \in \Sigma(1)} E^\rho_{\langle p, u\rangle}$, so that:
$$\dim H^0(X_\Sigma, \E)_u = \dim(\bigcap_{\rho \in \Sigma(1)} E^\rho_{\langle p, u\rangle}) = \rank \{ e \in \M \mid u \in P_{\vv(e), \sigma}\}.$$
Thus, we get the following formula for the Euler characteristic $\chi(X_\Sigma, \E)_u$:
\begin{equation}   \label{equ-chi-matroid-polyhedral}
\chi(X_\Sigma, \E)_u = \sum_{\sigma \in \Sigma} (-1)^{\codim(\sigma)} \rank\{e \in \M \mid u \in P_{\vv(e), \sigma} \}.
\end{equation}
We note that this formula is a mix of matroidal and polyhedral data. 

More generally, the $T$-equivariant maps from a $T-$linearized line bundle $\O(\psi)$, corresponding to a piecewise-linear function $\psi \in \PL(\Sigma, \Z)$, to $\E$ can be computed with Klyachko spaces:
$$\Hom_{X_\Sigma}^T(\O(\psi), \E) = E^{\rho_1}_{\psi(\bv_1)} \cap \ldots \cap E^{\rho_n}_{\psi(\bv_n)}$$
As a consequence, the entire Cox module of $\E$ has an expression in terms of Klyachko spaces:
$$M(\E) = \bigoplus_{\br \in \Z^n} E^{\rho_1}_{r_1} \cap \ldots \cap E^{\rho_n}_{r_n},$$
where the $i$-th generator of the Cox ring of $X_\Sigma$ acts by the expected rule:  $x_i \cdot E^{\rho_1}_{r_1} \cap \cdots \cap E^{\rho_n}_{r_n} \subseteq E^{\rho_1}_{r_1} \cap \cdots\cap E^{\rho_i}_{r_i-1} \cap \cdots \cap E^{\rho_n}_{r_n}$.

\subsection{The sheaf of sections of a tropical toric vector bundle} \label{subsec-sheaf}

Now we observe that the expressions in the previous section all make sense for any matroid vector bundle $\EE$ over a smooth, complete toric variety $X_\Sigma$.  As with vector bundles, the basic building blocks are the Klyachko flats $F^{\rho}_r = \{e \mid \vv(e)[p] \geq r\} \subset \M$.  We define the matroids $H^0(U_\sigma, \EE\!\!\mid_{U_\sigma})_u$, $H^0(X_\Sigma, \EE)_u$, and $M(\EE)$ in analogy with the vector bundle case. In particular, the Cox matroid of $\EE$ is the direct sum:

$$M(\EE) = \bigoplus_{\br \in \Z^n} F^{\rho_1}_{r_1}\cap \ldots \cap F^{\rho_n}_{r_n}.$$

\noindent
The Cox matroid is naturally a submatroid of the infinite matroid $\bigoplus_{\br \in \Z^n} \M$. For very large $r_i$, the flat $F^{\rho_i}_{r_i}$ is empty. Similarly, if $r_i$ is sufficiently negative, $F^{\rho_i}_{r_i} = \M$, so most of the summands of $M(\EE)$ are empty or $\M$. There is a natural inclusion $F^{\rho_1}_{r_1}\cap \ldots \cap F^{\rho_i}_{r_i} \cap \ldots \cap F^{\rho_n}_{r_n} \subseteq F^{\rho_1}_{r_1}\cap \ldots \cap F^{\rho_i}_{r_i-1} \cap \ldots \cap F^{\rho_n}_{r_n}$. The latter can be viewed as an action of the monoid $\Z_{\leq 0}^n$ on $M(\E)$.  In the representable case, this lifts to an action by the Cox ring of $X_\Sigma$.  The matroid of global sections $H^0(X_\Sigma, \EE)$ is naturally the submatroid of $\M(\EE)$ composed of those summands $F^{\rho_1}_{r_1}\cap \ldots \cap F^{\rho_n}_{r_n}$ where $r_i = \langle \bv_{\rho_i}, u\rangle$ for some $u \in M$.

For a face $\sigma \in \Sigma$ with $\sigma(1) = \{\rho_1, \ldots, \rho_d\}$, the \emph{matroid fiber} $\EE_\sigma$ over the torus fixed point corresponding to $\sigma \in \Sigma$ is the sum of quotient matroids $\bigoplus_{\br \in \Z^d} F^{\rho_1}_{r_1} \cap \cdots \cap F^{\rho_d}_{r_d} / \Span(\bigcup_{i =1}^d F^{\rho_1}_{r_1} \cap \cdots \cap F^{\rho_i}_{r_i +1}\cap \cdots \cap F^{\rho_d}_{r_d})$.  Recall that for a flat $F \subset \M$, the quotient $\M/F$ is the matroid on the complement $\M \setminus F$ where a subset $S$ is declared a basis if there is a basis $T$ of $F$ such that $S \cup T$ is a basis of $\M$.  While it is unclear to us what the quotient map $H^0(U_\sigma, \EE\!\!\mid_{U_\sigma}) \to \EE_\sigma$ might mean, we can still make sense of the notion of surjectivity for this map.  

\begin{definition}
    We say that $S \subset H^0(U_\sigma, \EE\!\!\mid_{U_\sigma})$ generates $\EE_\sigma$ if the span of the union of subsets $S_\br = S \cap F^{\rho_1}_{r_1} \cap \cdots \cap F^{\rho_d}_{r_d} \setminus \Span(\bigcup_{i =1}^d F^{\rho_1}_{r_1} \cap \cdots \cap F^{\rho_i}_{r_i +1}\cap \cdots \cap F^{\rho_d}_{r_d})$ spans $\EE_\sigma$. 
\end{definition}

For $e \in \M$, the class $[e] \in \EE_\sigma$ is defined by the corresponding copy of $e$ in $F^{\rho_1}_{\vv(e)(\bv_1)} \cap \cdots \cap F^{\rho_d}_{\vv(e)(\bv_d)}$.  


\begin{lemma}\label{lem-gradedbasisadapted}
    Let $\overline\B \subset \EE_\sigma$ be a basis, then the corresponding elements $\B \subset \M$ form a basis of $\M$ and $\Phi(\sigma) \subset A_\B$.  Moreover, any basis for which $\Phi(\sigma) \subset A_\B$ determines a basis of $\EE_\sigma$.  
\end{lemma}

\begin{proof}
    Suppose that $\Phi(\sigma) \subset A_{\B'}$ for a basis $\B' \subset \M$. Then $\B' \cap F^{\rho_1}_{r_1} \cap \cdots \cap F^{\rho_d}_{r_d}$ is a basis for any $\br \in \Z^d$. An element $b' \in \B'$ then appears in precisely one summand of $\EE_\sigma$, so $\bar\B'$ is a basis.  This shows that the rank of $\EE_\sigma$ coincides with the rank of $\M$. 

    Now let $\B \subset \M$ define a basis $\bar\B \subset \EE_\sigma$. We can write $\B$ as a disjoint union of the sets $\B_\br$, where the classes $\bar\B_\br$ give a basis of $F^{\rho_1}_{r_1} \cap \cdots \cap F^{\rho_d}_{r_d} / \Span(\bigcup_{i =1}^d F^{\rho_1}_{r_1} \cap \cdots \cap F^{\rho_i}_{r_i +1}\cap \cdots \cap F^{\rho_d}_{r_d})$.  
           
      We claim that $\bigcup_{\bs \geq \br} \B_\bs$ spans $F^{\rho_1}_{r_1} \cap \cdots \cap F^{\rho_d}_{r_d}$, where $\bs \geq \br$ is component-wise. First, if $\br$ has the property that $\Span(\bigcup_{i =1}^d F^{\rho_1}_{r_1} \cap \cdots \cap F^{\rho_i}_{r_i +1}\cap \cdots \cap F^{\rho_d}_{r_d}) = \emptyset$, then clearly $\B_\br$ is a basis for $F^{\rho_1}_{r_1} \cap \cdots \cap F^{\rho_d}_{r_d}$. As $F^{\rho}_r = \emptyset$ for $r > > 0$, this forms the base case of induction on the intersection lattice of the Klyachko flats.  Now suppose that the statement holds for all $\B_\bs$ with $\bs > \br$; this implies that $\bigcup_{\bs > \br} \B_\bs$ spans  $\Span(\bigcup_{i =1}^d F^{\rho_1}_{r_1} \cap \cdots \cap F^{\rho_i}_{r_i +1}\cap \cdots \cap F^{\rho_d}_{r_d})$. This assumption, taken with the fact that $\bar\B_\br$ is a basis of $F^{\rho_1}_{r_1} \cap \cdots \cap F^{\rho_d}_{r_d} / \Span(\bigcup_{i =1}^d F^{\rho_1}_{r_1} \cap \cdots \cap F^{\rho_i}_{r_i +1}\cap \cdots \cap F^{\rho_d}_{r_d})$ implies that $\bigcup_{\bs \geq \br} \B_\bs$ spans $F^{\rho_1}_{r_1} \cap \cdots \cap F^{\rho_d}_{r_d}$.  If $r < < 0$ then $F^\rho_r = \M$, so we have shown that $\B$ spans $\M$.  As $\bar\B$ is a basis of $\EE_\sigma$, it follows that $\B$ must also be a basis of $\M$.  
      
\end{proof}

The values of $\vv$ on $\B$ as in Lemma \ref{lem-gradedbasisadapted} over $\sigma$ must match a (possibly multi)set of characters $u_1, \ldots, u_r$.  We let $u(e_i)$ be the character of $e_i \in \B$.  The polyhedra $P_{\vv(e), \sigma}$ and $P_{\vv(e)}$ for $e \in \M$ still make sense as defined, and we have $H^0(U_\sigma, \EE\!\!\mid_{U_\sigma})_u = \{e \mid u \in P_{\vv(e), \sigma}\} \subset \M$ and $H^0(X_\sigma, \EE)_u = \{e \mid u \in P_{\vv(e)}\} \subset \M$ as before. 

\begin{lemma}\label{lem-basissurvives}
    Fix $u \in M$, and let $e \in H^0(X_\Sigma, \EE)_u$, then the following are equivalent:
    \begin{enumerate}
    \item $e$ defines an element in $\EE_\sigma$.
    \item $u$ is the vertex of $P_{\vv(e)}$ in the $\sigma$ direction.
    \end{enumerate}
    In this case, $u = u(e_i)$ for some basis member $e_i \in \B$ for any basis for which $\Phi(\sigma) \subset A_\B$.
\end{lemma}

\begin{proof}
    The fact that $e \in H^0(X_\Sigma, \EE)_u$ implies that the linear function defined by $u$ bounds $\vv(e)$ from below.  This is equivalent to $\langle \bv_\rho, u\rangle \leq \vv(e)[\bv_\rho]$ holding for all $\rho \in \sigma(1)$. The statement $(1)$ then means that these are equalities, which is equivalent to $(2)$.  Finally, $e$ must lie in the complement $F^{\rho_1}_{r_1} \cap \cdots \cap F^{\rho_d}_{r_d} \setminus \Span(\bigcup_{i =1}^d F^{\rho_1}_{r_1} \cap \cdots \cap F^{\rho_i}_{r_i +1}\cap \cdots \cap F^{\rho_d}_{r_d})$.  It follows that $u = u(e_i)$ for any $e_i \in \B_\br$. 
\end{proof}

\begin{definition}  \label{def-gl-gen-tmb}
    We say that $\EE$ is \emph{globally generated} if the elements of $H^0(X_\sigma, \EE)$ suffice to define a basis in each $\EE_\sigma$.
\end{definition}

The next theorem is now immediate from Lemmas \ref{lem-basissurvives} and \ref{lem-gradedbasisadapted}.

\begin{theorem}\label{thm-gg}
    A tropical toric vector bundle $\EE$ is globally generated if and only if for each $\sigma$ there is a basis $\B \subset \M$ such that the characters $u_i$ for the fiber $\E_\sigma$ are the vertices in the $\sigma$ direction of the polyhedra $P_{\vv(e)}$ $e \in \B$.  
\end{theorem}

Let $\br = (r_1, \ldots, r_n) \in \Z^n$, and let $\mathcal{L}$ denote the corresponding $T$-linearized line bundle bundle on $X_\Sigma$.  We may define the tensor product $\EE \otimes \mathcal{L}$ to be the tropical toric vector bundle with the same matroid as $\EE$ and diagram the matrix obtained by adding $r_i$ to the $i$-th row of the diagram of $\EE$ for each $1 \leq i \leq n$.  This is shown to coincide with the corresponding operation on toric vector bundles in \cite{KM-PL}.  Next we show that tropical toric vector bundles can be made globally generated by tensoring with a sufficiently high power of an ample line bundle.  
\begin{theorem}\label{thm-makegg}
    Let $\EE$ be a tropical toric vector bundle, and let $\mathcal{L}$ be a $T$-linearized ample line bundle on $X_\Sigma$, then there is an $N_0 > 0$ depending on $\EE$ and $\O(\psi)$ such that the bundle $\EE \otimes \mathcal{L}$ is globally generated for all $N \geq N_0$.  
\end{theorem}

\begin{proof}
    Let $D$ be the diagram of $\EE$.  Suppose that $e \in \M$ is part of an adapted basis over $\sigma \in \Sigma$, then the $\sigma(1)$ entries of the $e$-th column of $D$ coincide with the inner products $\langle \bv_\rho, m(e)\rangle$, for $\rho \in \sigma(1)$. It follows that if the $e$-th column of $D$ defines an ample class over $X_\sigma$, $m(e)$ is a vertex of $P_{\vv(e)}$. Moreover, if every column of $D$ defines an ample class, the criteria of \ref{thm-gg} must be satisfied, as every $\sigma$ has an adapted basis among the elements of $\M$.  Now the theorem follows from the fact that any divisor on $X_\Sigma$ can be made ample by tensoring with a sufficiently high multiple of $\mathcal{L}$. 
\end{proof}

\begin{remark}
    In the of proof of Theorem \ref{thm-makegg} we have shown that if every column of the diagram of a tropical toric vector bundle defines an ample class, then $\EE$ is globally generated. 
\end{remark}

\subsection{The equivariant Euler characteristic}\label{subsec-euler}

We finish this section by defining and discussing the equivariant Euler characteristic of a tropical toric vector bundle. 

\begin{definition}[Equivariant Euler characteristic of a tropical toric vector bundle]   \label{def-equiv-Euler-matriod-vb}
For a character $u$ we define the Euler characteristic $\chi(X_\Sigma, \EE)_u$ by:
\begin{align*}
\chi(X_\Sigma,\EE)_u &= \sum_{\sigma \in \Sigma} (-1)^{\codim(\sigma)} \rank H^0(U_\sigma, \EE\!\!\mid_{U_\sigma})_u,\\
&= \sum_{\sigma \in \Sigma} (-1)^{\codim(\sigma)} \rank\;\{e \in \M \mid u \in P_{\vv(e), \sigma} \}.    
\end{align*}
{Moreover, we let:
$$\chi(X_\Sigma,\EE) = \sum_{u \in M} \chi(X_\Sigma,\EE)_u.$$}
\end{definition}

{Definition \ref{def-equiv-Euler-matriod-vb} is made to coincide with the Euler characteristic of a toric vector bundle when the matroid is a ``DJS matroid" as in \cite{DJS} and Definition \ref{def-DJSmatroid-toricvectorbundle}.  

\begin{proposition}\label{prop-Euler-sheaf-tropicalization}
   Let $\E$ be a toric vector bundle over a toric variety $X_\Sigma$ determined by the data $(L, \Phi)$, and let $\EE = (\M(L), \Phi)$ be its tropicalization. If $\M(L)$ is an extension of a DJS matroid of $\E$, then $\chi(X_\Sigma, \E)_u = \chi(X_\Sigma, \EE)_u$ and $H^0(X_\Sigma, \E)_u = H^0(X_\Sigma, \EE)_u$.  
\end{proposition}
}

We will show Theorem \ref{thm-intro-sections}, which we restate below. 
\begin{theorem}\label{thm-EulerSections}
    Let $\EE$ be a tropical toric vector bundle on a smooth, projective toric variety $X_\Sigma$, and let $\mathcal{L}$ be an ample line bundle on $X_\Sigma$, then there is an integer $N_0 >0$ such that for all $N \geq N_0$ and $u \in M$ we have:\[\chi(X_\Sigma,\EE\otimes\mathcal{L}^{\otimes N})_u = \rank H^0(X_\Sigma,\EE\otimes\mathcal{L}^{\otimes N})_u.\]
    Moreover, for $N \geq N_0$, $\rank H^0(X_\Sigma, \EE\otimes \mathcal{L}^{\otimes N})$ is computed by an integral polynomial of degree $d=\dim(X_\Sigma)$ in $N$.
\end{theorem}

We prove Theorem \ref{thm-EulerSections} by proving that a more general identity holds among various functions on the real vector space $M_\R$.  Fix a matroid $\M$ and a smooth, projective fan $\Sigma$. We consider the set $\mathcal{F}(\Sigma, \M) \subset \Berg(\M)^n \subset \R^{m\times n}$ of $m \times n$ matrices which satisfy the tropical and apartment conditions to be a diagram of a tropical toric vector bundle for $\M$ over $X_\Sigma$. The integral points of $\mathcal{F}(\Sigma, \M)$ define tropical toric vector bundles over $X_\Sigma$.  

Fix $D \in \mathcal{F}(\Sigma, \M)$; this data determines a valuation $\vv$ and polyhedra $P_{\vv(e)}, P_{\vv(e), \sigma}$ for $e \in \M$. To emphasize the dependence on the diagram $D$ we denote these polyhedra by $P_{D, e}$ and $P_{D, e, \sigma}$, respectively. In particular, the column $d_i$ of $D$ corresponding to $e_i \in \M$, and the the polytope $P_{D,e_i} \subset M_\R$ is defined by the inequalities $\langle u, \bv_j \rangle \leq d_{ij}$.  We let $I_{D,e_i}: M_\R \to \R$ denote the corresponding indicator function.  It is possible that $I_{D,e_i}(u) = 0$ for all $u \in M_\R$, e.g. if the column $d_i$ does not define an effective divisor. 

Similarly, for a set $S \subset \M$, we let $I_{D,S}$ be the indicator function of the intersection $P_{D, S}$ of the $P_{D,e_i}$ with $e_i \in S$.  This operation corresponds to taking the row-wise $\min$ of the $d_i$ for $e_i \in S$.  It is straightforward to check that $I_{D,S} = I_{D,\Span(S)}$, so we only even consider support functions corresponding to flats of $\M$.

The global sections functor defines a real-valued function $h_D^0:M_\R \to \R$, where $h_D^0(u) = \rank H^0(X_\Sigma, \EE)_u = \rank F^{\rho_1}_{\langle u, \bv_1 \rangle}\cap \cdots \cap F^{\rho_1}_{\langle u, \bv_n \rangle}$.  We express $h_D^0$ as a linear combination of functions of the form $I_{D,F}$ for flats $F \subset \M$. 

Observe that $I_{D,F}(u) = 1$ only if $F^{\rho_1}_{\langle u, \bv_1 \rangle}\cap \cdots \cap F^{\rho_1}_{\langle u, \bv_n \rangle} \supseteq F$, and $F = F^{\rho_1}_{\langle u, \bv_1 \rangle}\cap \cdots \cap F^{\rho_1}_{\langle u, \bv_n \rangle}$ precisely when $I_{D,F}(u) = 1$ but $I_{D,G}(u) = 0$ for all $G \supsetneq F$. Let $T_{D,\geq r}$ denote the set of $u \in M_\R$ where $h_D^0(u) \geq r$, and we let $I_{D,\geq r}$ be the indicator function for $T_{D,\geq r}$.  Finally, let $J_r$ be the set of flats of rank $r$.

\begin{lemma}\label{lem-globalsections}
For any $D \in \mathcal{F}(\Sigma, \M)$ we have:\\

\[I_{D,\geq r} = \sum_{\emptyset \neq B \subseteq J_r} (-1)^{|B|+1} I_{D,\Span\{\cup F_j \mid j \in B\}}.\]  

\end{lemma}

\begin{proof}
    One of the functions $I_{D,\Span\{\cup F_j \mid j \in B\}}(u) \neq 0$ only if $I_{D, F}(u) \neq 0$ for some flat of rank $\geq r$.  This in turn implies that $h_D^0(u) \geq r$, so the support of the right hand side is contained in $T_{D, \geq r}$. Now suppose $u \in T_{D, \geq r}$, this implies that $I_{D, F}(u) = 1$ for all rank $r$ flats $F \in J_r$ contained in $F^{\rho_1}_{\langle u, \bv_1 \rangle}\cap \cdots \cap F^{\rho_1}_{\langle u, \bv_n \rangle}$ and moreover that $I_{D, G}(u) = 0$ for any flat containing a rank $r$ flat which is not contained in $F^{\rho_1}_{\langle u, \bv_1 \rangle}\cap \cdots \cap F^{\rho_1}_{\langle u, \bv_n \rangle}$.  Let us denote the former by $F_1, \ldots, F_\ell$. Evaluating the right hand side at $u$, we only pick up non-zero terms among the summands built from the $F_1, \ldots, F_\ell$.  These terms sum to $1$.
\end{proof}

Now we define $T_{D,r} = T_{D,\geq r} \setminus T_{D,\geq r+1}$ to be the set of points where $h_D^0(u) = r$, and let $I_{D, r} = I_{D, \geq r} - I_{D, \geq r+1}$ be the indicator function of $T_{D, r}$.

\begin{corollary}\label{cor-globalsections}
With $D \in \mathcal{F}(\Sigma, \M)$ as above we have:\\

\[h_D^0 = \sum_{r = 1}^{\rank(\M)} rI_{D,r}.\]
    
\end{corollary}

\noindent
The above expression for $h_D^0$ can be rewritten as a sum of indicator functions for the polyhedra $P_{D, F}$:\[h_D^0 = \sum_{F \subseteq \M} c_F I_{D, F}.\] Notably, by definition the coefficients $c_F$ are integral and depend only on the matroid $\M$.  Let $F$ be a flat, and let $\Upsilon(F)$ be the set of sets $S$ of flats such that the elements of $S$ are all flats of the same rank and the span of the union of the flats in $S$ is $F$. Then it is straightforward to show from the definition of $I_r$ that $c_F = \sum_{S \in \Upsilon(F)} (-1)^{|S|+1}$. 

The Khovanskii-Pukhlikov theory of integration on convex chains \cite{Khovanskii-Pukhlikov-1, Khovanskii-Pukhlikov-2} implies that the function which counts the number of lattice points in $P_{D, F}$ is piecewise polynomial in the entries of $D$ when $D$ lies in the set of diagrams for which $P_{D, F}$ has outer normal fan $\Sigma$. This is a consequence of \cite[Theorem 1]{Khovanskii-Pukhlikov-1} and the fact that the half-space inequalities defining $P_{D, F}$ depend in a piecewise-linear way on the entries of $D$. Let $\mathcal{F}^+(\Sigma, \M) \subset \mathcal{F}(\Sigma, \M)$ be the subset of diagrams for which each $P_{D, F}$ has outer normal fan $\Sigma$.  Let $\mathfrak{h}_{\Sigma, \M}: \mathcal{F}^+(\Sigma, \M) \to \Z$ be the function $\mathfrak{h}_{\Sigma, \M}(D) = \sum_{u \in M} h^0_D(u)$.

\begin{corollary}\label{cor-ppsections}
    For a matroid $\M$ and fan $\Sigma$, the function $\mathfrak{h}_{\Sigma, \M}$ is piecewise polynomial of degree $\dim(X_\Sigma)$ in the entries of $D$. 
\end{corollary}

\begin{example}
 We compute $\mathfrak{h}_{\P^2, U_{2,3}}$.  Let $\{x_0, x_1, x_2\}$ be the ground set of $U_{2,3}$, then the flats are $\emptyset, \{x_0\}, \{x_1\}, \{x_2\}, \{x_0, x_1, x_2\}$. The coefficients $c_F$ for these flats are $c_\emptyset = 0$, $c_{x_0} = c_{x_1} = c_{x_2} = 1$, and $c_{x_0,x_1,x_2} = -1$. 

 The set of diagrams $\mathcal{F}(\P^2,U_{2,3})$ is precisely the set of $3\times 3$ matrices $D = [d_{ij}]$ with the property that the minimum of any row $\{d_{i0}, d_{i1}, d_{i2}\}$ occurs twice. In other words, $\mathcal{F}(\P^2,U_{2,3})$ is the support of the fan $\Berg(U_{2,3})^3 \subset \R^{3\times 3}$. The set $\mathcal{F}^+(\P^2,U_{2,3})$ is composed of those $D$ where $P_{D, x_0}, P_{D, x_1}, P_{D, x_2}$, and $P_{D, \{x_0,x_1,x_2\}}$ are all ample with respect to the fan of $\P^2$.  This occurs when the sums of columns $d_{0j} + d_{1j} + d_{2j}$ and the sum of the minimal entries of the rows $\sum_{i=0}^2 \min\{d_{i0}, d_{i1}, d_{i2}\}$ are positive.  The number of integral points in the polytope associated to $a, b, c$ with $a + b + c > 0$ is $\binom{a+b+c+2}{2}$. Putting this together, we see that 

 \[\mathfrak{h}_{\P^2, U_{2,3}}(D) = \sum_{j = 0}^2 \binom{d_{0j} + d_{1j} + d_{2j} + 2}{2} - \binom{\sum_{i=0}^2 \min\{d_{i0}, d_{i1}, d_{i2}\} +2}{2}.\]
\end{example}

Now we fix a face $\sigma \in \Sigma$ and consider the polyhedron $P_{D, e_i, \sigma} \subset M_\R$ defined by the inequalities $\langle u, \vv_j \rangle \leq d_{ij}$, where $j \in \sigma(1)$. The indicator functions  $I_{D,F,\sigma}$, $I_{D,\geq r,\sigma},$ and $I_{D,r,\sigma}$ are defined accordingly. We let $h^0_{D,\sigma} = \sum_{r = 1}^{\rank(\M)} rI_{D,r, \sigma}$. As before, we may rewrite this sum in terms of the indicator functions $I_{D, F, \sigma}$: \[h^0_{D, \sigma} = \sum_{F \subset \M} c_F I_{D, F, \sigma}.\] We let $\chi_D(u) = \chi(X_\Sigma, \EE)_u$.

\begin{lemma}\label{lem-affineglobalsections}
    For $u \in M$ we have $h^0_{D,\sigma}(u) = \rank H^0(U_\sigma, \EE_D\mid_{U_\sigma})_u$. Moreover, we have:\[\chi_D = \sum_{\sigma \in \Sigma} (-1)^{\codim(\sigma)} h^0_{D,\sigma}.\]
\end{lemma}

\begin{proof}
    The proof that $h^0_{D,\sigma}(u) = \rank H^0(U_\sigma, \EE_D\mid_{U_\sigma})_u$ is similar to the proof of Lemma \ref{lem-globalsections}.  The identity $\chi_D = \sum_{\sigma \in \Sigma} (-1)^{\codim(\sigma)} h^0_{D,\sigma}$ follows from the definition of the Euler characteristic.
\end{proof}

Let $\Delta$ be a polytope with outer normal fan equal to $\Sigma$; in particular there are $d_j$ for which $\Delta = \{u \mid \langle u, p_j \rangle \leq d_j, j \in \Sigma(1)\}$. Similarly, for $\sigma \in \Sigma$ let $\Delta_\sigma = \{u \mid \langle u, p_j \rangle \leq d_j, j \in \sigma(1)\}$, and let $I_\Delta$ and $I_{\Delta, \sigma}$ be the indicator functions of these polyhedra.  For any convex polytope $\Delta$ with outer normal fan $\Sigma$ the Brianchon-Gram formula (see \cite[p. 297--303]{Grunbaum}) gives:\[I_\Delta = \sum_{\sigma \in \Sigma} (-1)^{\codim(\sigma)} I_{\Delta, \sigma}.\]

\begin{proposition}\label{prop-eulerchiprop}
    Let $D \in \mathcal{F}^+(\Sigma, \M)$, then $\chi_D = h^0_D$. 
\end{proposition}

\begin{proof}
    We compute: \[h^0_D = \sum_{F \subseteq \M} c_F I_{D, F} = \sum_{F \subseteq \M} c_F (\sum_{\sigma \in \Sigma} (-1)^{\codim(\sigma)} I_{D, F, \sigma})\] \[= \sum_{\sigma \in \Sigma} (-1)^{\codim(\sigma)} (\sum_{F \subseteq \M} c_F I_{D, F, \sigma}) = \sum_{\sigma \in \Sigma} (-1)^{\codim(\sigma)} h^0_{D, \sigma} = \chi_D.\]
\end{proof}

Now let $D \in \mathcal{F}(\Sigma, \M)$, and let $\mathcal{L}$ be an ample line bundle on $X_\Sigma$.  We write the moment polytope of $\mathcal{L}$ as $\Delta_\mathcal{L} = \{ u \mid \langle u, p_j\rangle \leq r_j\}$.  In particular, we use the convention that $\Sigma$ is the outer normal fan of $\Delta_\mathcal{L}$. The diagram of $\EE_D \otimes \mathcal{L}^N$ has entries $d_{ji} + Nr_j$. We let $D + N\br$ denote this diagram.

\begin{lemma}\label{lem-flatsample}
    For $D$ and $\mathcal{L}$ as above, there is an $N_0$ such that for all $N \geq N_0$, the polyhedra $P_{D + N\br, F}$ all have outer normal fan $\Sigma$. 
\end{lemma}

\begin{proof}
    First we observe that $\Delta_{D + N\br, F} = \{u \mid \langle u, p_j\rangle \leq \min\{d_{ji} \mid \e_i \in F\} + Nr_j\}$.  It follows that there is a number $N_F$ such that for all $N \geq N_F$, $P_{D + N_F\br,F}$ has outer normal fan $\Sigma$.  Now we take $N_0$ to be the maximum of the $N_F$ for flats $F \subset \M$. 
\end{proof}

Theorem \ref{thm-EulerSections} now follows by taking $N_0$ to be the bound from Lemma \ref{lem-flatsample}. We also obtain the following from Corollary \ref{cor-ppsections}. This completes the proof of Theorem \ref{thm-intro-sections}.

\begin{corollary}\label{cor-ppsections2}
    For $D \in \mathcal{F}(\Sigma, \M)$ with corresponding tropical toric vector bundle $\EE$, and an ample line bundle $\mathcal{L}$, the function which computes the number of global sections of the bundle $\E\otimes \mathcal{L}^{\otimes N}$ is eventually a polynomial of degree $\dim(X_\Sigma)$ in $N$.
\end{corollary}

\section{Tautological tropical toric vector bundles on the permutahedral variety}   \label{sec-tautological-matroid-bundles}

In this section we see that each matroid $\M$ comes with a canonical tropical toric vector bundle $\EE_\M$ on the permutahedral toric variety and we prove Theorem \ref{thm-main-tautological}. Recall that the $m$-dimensional \emph{permutahedron} is the convex hull:
$$P_m = \conv\{ \pi(1, \ldots, m) \mid \pi \in S_m\}.$$
The \emph{permutahedral fan} $\Sigma_m$ is the normal fan of the permutahedron. The corresponding toric variety $X_{\Sigma_m}$ is the \emph{permutahedral toric variety}.


Let $\M$ be a loop-free matroid with rank $r$ and $|\M| = m$. 

\begin{proposition}  \label{prop-perm-fan-Grobner-fan}
The permutahedral fan $\Sigma_m$ refines the Gr\"obner fan $\GF(\M)$.
\end{proposition}
\begin{proof}
The vertices of the matroid polytope $P_\M$ are among those of the permutahedron $P_m$.    
\end{proof}

\subsection{The map $\Phi_\M$ and the bundle $\EE_\M$}
Let $w \in \R^\M$. Consider the associated $\R$-filtration $(\M_{w\geq k})_{k \in \R}$ by flats on $\M$ where
$$\M_{w \geq k} = \Span\{ i \in \M \mid w_i \geq k\}.$$
The filtration $(\M_{w\geq k})_{k \in \R}$ determines a point $w' \in \Berg(\M) \subset \R^\M$ by:
$$w'_i = \sup\{k \mid i \in \M_{w \geq k} \}.$$

Thus we obtain a canonical projection map:
$$\Phi_\M: |\GF(\M)| \to \Berg(\M), \quad \Phi_\M(w) = w'.$$

Moreover, by definition $\Phi_\M$ is the identity when restricted to $\Berg(\M)$, so $\Phi_\M\circ \Phi_\M = \Phi_\M$. Now we can show that $\Phi_\M$ is in fact a piecewise linear map. For $\sigma \in \GF(\M)$ consider the map $\Phi_\sigma: \sigma \to A_{B_\sigma} = \sigma \cap \Berg(\M)$ defined as follows. For $e_i \in \M$ let $C_i$ be the circuit in $\{e_i\} \cup B_\sigma$ containing $e_i$. Then for $w=(w_1, \ldots, w_m) \in \sigma$ we put:
$$\Phi_\sigma(w)_i = 
\begin{cases} 
w_i,  \quad e_i \in B \\
\min\{w_j \mid j \in C_i \setminus \{i\} \}, \quad i \notin B_\sigma
\end{cases}$$
The following is straightforward to verify:
\begin{proposition}
The map $\Phi_\M$ restricted to the cone $\sigma$ coincides with $\Phi_\sigma$.
\end{proposition}

\begin{remark}
When $\M$ is a representable matroid, the canonical map $\Phi_\M$ is a special case of a general construction in Gr\"obner theory and tropical geometry (see \cite[Section 3.2]{Kaveh-Manon-SIAGA}). 
\end{remark}

We recall that the Gr\"obner fan of $\M$ is a complete fan which is refined by the permutahedral fan $\Sigma_m$. Thus we can consider $\Phi_\M$ as a piecewise linear map from $|\Sigma_m|$ to $\Berg(\M)$.

\begin{definition}[Tautological tropical toric vector bundle of a matroid]  \label{def-tautological-mb}
For a matroid $\M$, We call the tropical toric vector bundle $\EE_\M$ given by the piecewise linear map $\Phi_\M: |\Sigma_m| \to \Berg(\M)$, the \emph{tautological tropical toric vector bundle of $\M$}. 
\end{definition}

The function $\Phi_\M$ provides a mechanism to create toric tropical toric vector bundles.  For any matroid $\M$ and integral piecewise-linear functions $\psi_1, \ldots, \psi_m: N \to \Z$ we can find a fan $\Sigma$ and a tropical toric vector bundle $\EE$ over $X_{\Sigma}$ with matroid $\M$ where the functions $\vv(e_i)$ for $e_i \in \M$ are as ``close as possible" to the $\psi_i$. 

\begin{proposition} \label{prop-matroid-vb-upto-pull-back}
Let $\M$ be a matroid and $\psi_1, \ldots, \psi_n: N \to \Z$ be integral and piecewise-linear, and let $\Psi = (\psi_1, \ldots, \psi_m): N \to \R^m$, then there is a fan $\Sigma$ such that the composition $\Phi_\M\circ \Psi: N \to \Berg(\M)$ defines a toric tropical toric vector bundle over $X_\Sigma$. 

Moreover, if $(\psi_1, \ldots, \psi_m)$ defines a $\PL(|\Sigma|, \Z)$-valued point of $\Berg(\M)$, then $\vv(e_i) = \psi_i$ for this bundle.    
\end{proposition}

\begin{proof}
We let $\Sigma_i$ denote the polyhedral fan formed by the domains of linearity for the piecewise-linear function $\psi_i$, and we take $\Sigma_0$ to be a fan which refines each of the $\Sigma_i$.  Let $\Sigma$ be a fan which refines $\Sigma_0$ and the pullback of the permutahedral fan $\Sigma_m$ under $\Phi_\M\circ\Psi$, then it is straightforward to verify that $\Phi_\M\circ \Psi: |\Sigma| \to \Berg(\M)$ satisfies the criteria of Definition \ref{def-matroid-vb-1}. If $(\psi_1, \ldots, \psi_m)$ is a point of $\Berg(\M)$, then $\Psi(N) \subset \Berg(\M)$, so applying $\Phi_\M$ is identity, and in turn $\vv(e_i) = \psi_i$.  
\end{proof}

\begin{remark}
The $\PL(N, \Z)$-valued points on $\Berg(\M)$ are the analogue of toric vector bundles up to pull-back by toric morphisms (see Theorem \ref{th-tvb-PL-valuation}).
\end{remark}

Before we delve into the structure of $\EE_\M$ we need a lemma on various distinguished bases of the matroid $\M$ determined by a maximal face $\sigma \in \Sigma_m$.  Recall that we have set the convention that the initial form $\In_w(C)$ of a circuit $C$ with respect to a weighting (or term order) $w$ is the set of \emph{minimal} elements of $C$.  For $w \in \Q^m$, and a set $S \subset \M$ we let $\wt_w(S) = \sum_{e \in S} w(e)$.

A \emph{greedy basis} of a point $w \in \GF(\M)$ is constructed by induction on the flag of flats defined by $\Phi_\M(w)$.  In particular, the basis of $F^w_r$ is constructed from that of $F^w_{r+1}$ by adding a member of $F^w_r$ not in $F^w_{r+1}$ with the largest possible $w$-weight. Observe that if $w$ is general then this basis is unique. By construction the greedy basis is adapted to the filtration defined by $\Phi_\M(w)$. Moreover, if $F^w_{r+1} \neq F^w_r$, the new element added to the greedy basis must have weight $r$.  

\begin{lemma}\label{lem-bases}
    Fix a maximal face $\sigma \in \Sigma_m$ and a basis $\B$, the following are equivalent: 

    \begin{enumerate}
        \item $\B$ is the complement of those elements of $\M$ which can be realized as initial forms of circuits of $\M$ with respect to the order on the elements of $\M$ determined by $\sigma$ (equivalently, with respect to any general element of $\sigma$).
        \item $\B$ is the lex-maximal basis of $\M$ with respect to $\sigma$.
        \item The function $\wt_w$ is maximized among bases at $\B$.
        \item $\B$ is the greedy basis with respect to any general element of $\sigma$. 
    \end{enumerate}
    
\end{lemma}

\begin{proof}
    First we show that the set determined by $(1)$ is independent.  Let $C$ be a circuit which holds among the elements in the complement of the set of initial forms; then $\In_w(C)$ cannot be in the set determined by $(1)$, a contradiction.  Next, we show that the basis determined by $(4)$ is contained in $(1)$. Let $\B$ be the greedy  basis, and suppose $e \in \B$ is $\In_w(C)$ for some circuit.  Then $w(e) < w(c)$ for all $c \in C$; this implies that $\Phi_\M(w)(e) > w(e)$, which contradicts this property of the greedy basis. This shows both that $(1)$ determines a basis, and that this basis coincides with the one determined by $(4)$.  Now let $\B$ be the basis determined by $(1)$ and $\B'$ be the basis determined by $(3)$. Suppose $e \in \B$, $e \notin \B'$, then there is a circuit $C = \{e, D\}$ where $D \subset \B'$. We cannot have $e = \In_w(C)$, so there is an element $e' \in \B'$ with smaller $w$-weight than $e$ such that $B' \setminus \{e'\} \cup \{e\}$ is a basis.  This contradicts that $\B'$ has maximal total weight.  The proof that the bases determined by $(1)$ and $(2)$ are the same is identical.    
\end{proof}

Following Lemma \ref{lem-bases}, we adopt the convention that $\GF(\M)$ is the \emph{outer} normal fan of of the matroid polytope of $\M$.  In particular, the maximal faces of $\GF(\M)$ are in bijection with bases $\B$, and the basis $\B_C$ of a maximal face $C \in \GF(\M)$ is the set of tuples $w$ for which $\B$ is the basis of maximal total weight.  We let $\B_\sigma$ denote the basis associated to $\sigma \in \Sigma_m$. Recall that $\delta_i$ denotes the $i$-th standard basis vector of $\R^m$ and $\delta_S = \sum_{i \in S} \delta_i$ for $S \subset [m]$.  The following is the existence statement of Theorem \ref{thm-main-tautological}.

\begin{corollary}\label{cor-tautologicalGG}
    For $\sigma \in \Sigma_m$, the adapted basis for $\Phi_\M$ restricted to $\sigma$ is $\B_\sigma$. For $e_i \in \B_\sigma$, the character $m_\sigma(e_i)$ is $\delta_i$. The diagram $D(\EE_\M)$ is the matrix with $S, i$-th entry equal to $1$ if $e_i$ is in the span of the $e_j$ for $j \in S$, and $0$ otherwise.  The bundle $\EE_\M$ is globally generated. 
\end{corollary}

\begin{proof}
    The first two statements of this corollary are restatements of Lemma \ref{lem-bases} and the definition of $\Phi_\M$.  For the third statement, recall that the parliament member $P_{\vv(e)}$ for $e_i \in \M$ is the polyhedron of the divisor on $\Sigma_m$ determined by the $i$-th column of the diagram. As $\B_\sigma$ is the greedy basis for $\sigma$, for $\delta_S \in \sigma(1)$, we see that the $S$-th entry of the $j$-th column is $1$ if and only if $j \in S$; this shows that the character of $e_j$ at $\sigma$ is the indicator vector $\delta_j$.  Moreover, for any other ray $\delta_{S'}$, the inner product $\langle \delta_{S'}, \delta_j \rangle $ is $1$ if $j \in S'$ and $0$ otherwise, which is always less than or equal to the $S', j$-th entry of the diagram.  This shows that $\delta_j$ is a vertex of $P_{\vv(e_j)}$ when $e_j \in \B_\sigma$.     
\end{proof}

\subsection{Pullback under the Cremona transformation}\label{subsec-cremona}

We study an involution of $X_m$ defined from the inverse map $t \to t^{-1}$ on the torus $T$.  The inverse defines a linear map $\chi: N \to N$ which takes each maximal face $\sigma \in \Sigma_m$ isomorphically onto $\sigma^\vee$. In particular, $\chi(\delta_S) = -\delta_S = \delta_{S^c} - \delta_{[m]}$.  The induced map $\phi_\chi: X_m \to X_m$ is the \emph{Cremona} transformation of the permutahedral variety. 

\begin{proposition}\label{prop-chipullback}
      For the bundle $\phi_\chi^*\EE_\M$, the basis $\B_{\sigma^\vee}$ is an adapted basis of the face $\sigma \in \Sigma_m$.  For $e_j \in \B_{\sigma^\vee}$ the character $m_\sigma(e_j)$ is $-\delta_j$.  The diagram $D(\phi_\chi^*\EE_\M)$ is the matrix with $S, j$-th entry equal to $0$ if $e_j$ is in the span of the set $\{e_k \mid k \in S^c\}$ and $-1$ otherwise. The bundle $\phi_\chi^*\EE_\M$ is globally generated.
\end{proposition}

\begin{proof}
    The bundle $\phi_\chi^*\EE_\M$ is defined by the function $\Phi_\M\circ \chi: \Sigma_m \to \Berg(\M)$. It follows that the face $\sigma \in \Sigma_m$ is mapped into the apartment corresponding to $\B_{\sigma^\vee}$. The map $\Phi_\M\circ \chi$ takes $\delta_S$ to the weight vector of the flat defined by the $0$-entries of $-\delta_S$. The latter is precisely the span of those elements in $S^c$.  This implies the formula for $D(\phi_\chi^*\EE_\M)$.  We have that $\langle \delta_S, m_\sigma(e_j)\rangle$ is $0$ when $j \in S^c$ and $-1$ otherwise, this means that $m_\sigma(e_j) = -\delta_j$.  We observe that pullback of any globally generated bundle under a linear map will be globally generated, but we can also see directly that the inner product of $-\delta_j$ with any ray generator $u_S$ is always less than or equal to the corresponding entry of $D(\phi_\chi^*\EE_\M)$.  This means that $-\delta_j$ is a vertex of the polyhedron $P_{\vv(e_j)}$ in the parliament of $\phi_\chi^*\EE_\M$.   
 \end{proof}

Now observe that if we apply $\chi$ to $\Phi_{\M^\vee}$ we obtain a bundle $\phi_\chi^*\EE_{\M^\vee}$ with adapted basis $\B_{\sigma^\vee}(\M^\vee)$ over the face $\sigma$ with characters $-\delta_j$.  By Lemma \ref{lem-bases}, $\B_{\sigma^\vee}(\M^\vee)$ is the \emph{minimally weighted} basis with respect to a general element of $\sigma$, in particular it is the basis labelled by the complement of the indices of $\B_\sigma(\M)$ in $[m]$.

As a direct consequence of Corollary \ref{cor-tautologicalGG} and \ref{subsec-equiv-Chern-class}, the $K$-class $[\phi_\chi^*\EE_{\M^\vee}] \in K_0^T(X_m)$ and Chern classes $c_i^T(\phi_\chi^*\EE_{\M^\vee}) \in H^i_T(X_m)$ are the tautological $K$-class $[Q_\M]$ and Chern classes $c_i^T(Q_\M)$ defined in \cite[Definition 3.9]{BEST}.
Moreover, $[\EE_\M] \in K_0^T(X_m)$ and $c_i^T(\EE_\M) \in H^i_T(X_m)$ are $[S_\M^\vee]$ and $c_i^T(S_\M^\vee)$.  This proves the characteristic class statement in Theorem \ref{thm-main-tautological}.

\subsection{Relationship with the bundles $S_L$ and $Q_L$}
In this section we make a connection with the tautological bundles and the toric vector bundles constructed in \cite{BEST}.  This requires to switch the convention of how a matroid is associated to a linear ideal $L$. Let $L$ be a linear subspace of $\k^m$.  We let $\M^\vee(L)$ be the matroid on $e_i$ $1 \leq i \leq m$ whose circuits are the minimally supported members of $L$.  We let $\M(L)$ denote the matroid whose bases arise as the sets of initial forms for weights $w \in \R^m$ with distinct entries. Note that this is dual to the convention used in the rest of this paper. 

In terms of Gr\"obner theory, the bases of $\M^\vee(L)$ arise as the standard monomial bases with respect to these weights. Standard monomials and initial forms for a fixed weight are complementary, so $\M(L)$ and $\M^\vee(L)$ are in fact dual matroids. 

For $u \in M$ let $\O_u$ be the $T$-linearization of the trivial line bundle corresponding to the character $u$. Following \cite{BEST} we define $S^\circ_L \subset \bigoplus_{i =1}^m \O_{\delta_i}$ to be the unique toric subbundle with general fiber $L \subset \k^m$, and $Q^\circ_L$ to be the associated quotient. We have an exact sequence:
$$0 \to S^\circ_L \to \bigoplus_{i = 1}^m \O_{\delta_i} \to Q^\circ_L \to 0.$$
To place these bundles in our treatment of tautological bundles we require the following Lemma. 
\begin{lemma}\label{lem-phiproperty}
Let $\E = \bigoplus_{i = 1}^m \mathcal{L}_{\psi_i}$ be a split bundle on a toric variety $X_\Sigma$, where $\psi_1, \ldots, \psi_m \in \PL(\Sigma, \Z)$ define $T$-linearized divisors on $X_\Sigma$.  Suppose that $\F$ is a toric vector bundle and a quotient of $\E$, and that the general fibers of these bundles are related by $F = E/L$ for $L \subset E$ a linear subspace of dimension $r$.  The piecewise-linear map $\Phi_\F: |\Sigma| \to \Trop(L)$ is then computed by the composition of $\Phi_{\M^\vee(L)}$ with the linear map $N \to \Z^m$ given by $p \to (\psi_1(p), \ldots, \psi_m(p))$. 
\end{lemma}

\begin{proof}
    This is essentially a special case of \cite[Proposition 4.5]{KM-PL} and Proposition \ref{prop-matroid-vb-upto-pull-back}. Let $\rho \in \Sigma(1)$, then the fiber over the general point of the orbit corresponding to $\rho$ is the associated-graded space of the filtration on $E/L$ induced by the weight $(\psi_1(p), \ldots, \psi_m(p))$.  This is the induced filtration of $\Phi_{\M(L)}(w(\bv_\rho))$ by definition. These are the rows of the diagram of $\F$ and thus determine $\Phi_\F$.
\end{proof}

\begin{proposition}\label{prop-best1}
The tropicalization of the bundle $Q^\circ_L$ is $\EE_{\M^\vee(L)}$.
\end{proposition}

\begin{proof}
    By Lemma \ref{lem-phiproperty}, the piecewise linear map $\Phi_{Q^\circ_L}: \Sigma_m \to \Trop(L)$ is $\Phi_{\M^\vee(L)}$ composed with the identity map. 
\end{proof}

Following \cite{BEST}, we obtain new bundles $S_L, Q_L$ by pulling $S^\circ_L, Q^\circ_L$ back along the Cremona transformation:

$$ 0 \to S_L \to \bigoplus_{i = 1}^m \O_{-\delta_i} \to Q_L \to 0.$$

\begin{proposition}\label{prop-best2}
The tropicalization of the bundle $Q_L$ is $\phi_\chi^*\EE_{\M^\vee(L)}$.
\end{proposition}

\begin{proof}
    This is a consequence of Proposition \ref{prop-best1}
\end{proof}

As a reminder, the basis of $\phi_\chi^*\EE_{\M^\vee(L)}$ at $\sigma$ is the revlex/minimal basis $\B_{\sigma^\vee}(\M^\vee(L))$ of $\M^\vee(L)$ with respect to $\sigma$.  This is the complement of the lex-first/maximal basis $\B_\sigma(\M(L))$ of $\M(L)$, and its characters are the negative indicator vectors.

Finally, we get dual bundles $S^\vee_L, Q^\vee_L$, and a corresponding exact sequence:

$$0 \leftarrow S_L^\vee \leftarrow \bigoplus_{i =1}^m \O_{\delta_i} \leftarrow Q_L^\vee \leftarrow 0.$$

\begin{proposition}
The tropicalization of the bundle $S_L^\vee$ is $\EE_{\M(L)}$. 
\end{proposition}

\begin{proof}
    The general fiber of $S_L^\vee$ is a quotient by a subspace whose matroid is $\M^\vee(L)$. Now by Lemma \ref{lem-phiproperty} the piecewise linear function $\Phi_{S^\vee_L}$ is $\Phi_{\M(L)}$ composed with the identity map.  
\end{proof}

As a reminder, the basis of $\EE_{\M(L)}$ at $\sigma$ is the lex-first/maximal basis $\B_\sigma(\M(L))$, and its characters are the indicator vectors.


\section{Matroid extensions and tropical toric vector bundles}
\label{sec-matroid-extension}
In Section \ref{subsec-tvb-PL-map-trop} we discussed how to construct toric vector bundles from tropical points on a linear ideal $L$. Ultimately, this construction only depends on the vector space $E = \k^n/L$.  In particular, if we extend the underlying matroid $\M(L)$ by adding vectors from $E$ the resulting toric vector bundle does not change.  This is important as it is possible for a toric vector bundle $\E$ to be defined by a matroid $\M_1$ yet have behavior which is only seen in an extension $F: \M_1 \to \M_2$.  We explore the consequences of this flexibility in the nonrepresentable case. 

\subsection{Matroid extensions}
A matroid extension $\phi: \M_1 \to \M_2$ is a one-to-one map on the underlying ground sets such that the matroid induced on $\phi(\M_1)$ by $\M_2$ is $\M_1$.  We only consider extensions for matroids of equal rank.  The following is straightforward. 

\begin{proposition}\label{prop-extensions}
    Let $\M_1$ and $\M_2$ be matroids of rank $r$, then the following are equivalent. 

    \begin{enumerate}
        \item $\phi: \M_1 \to \M_2$ is an extension.
        \item $C \subset \M_1$ is a circuit if and only if $\phi C \subset \M_2$ is a circuit.
        \item $\B \subset \M_1$ is a basis if and only if $\phi \B \subset \M_2$ is a basis. 
    \end{enumerate}
\end{proposition}
We show that if $\EE$ is a tropical toric vector bundle with matroid $\M_1$, and $\phi:\M_1 \to \M_2$ is an extension, then $\EE$ defines a matroidal vector bundle with matroid $\M_2$. In keeping with the vector bundle case, we expect that these bundles should be regarded as isomorphic. 

Let $\phi: \M_1 \to \M_2$ be an extension, and let $F \subset \M_1$ be a flat.  We obtain a flat $\phi_*F \subset \M_2$ by taking the span of $\phi(F)$.  We can extend this operation to weighted flags of flats, in other words points on the Bergman fan $\Berg(\M_1)$.

\begin{proposition}
    Let $\phi: \M_1 \to \M_2$ be an extension, and let $w \in \Berg(\M_1)$, then the flats $F^{\phi_*w}_r = \phi_*F^w_r$ form an integral filtration of $\M_2$ by flats. The induced map $\phi_*: \Berg(\M_1) \to \Berg(\M_2)$ is a piecewise-linear isomorphism onto the union of those apartments in $\Berg(\M_2)$ coming from bases of $\M_1$.
\end{proposition}

\begin{proof}
    It is clear that $F^{\phi_*w}_r \supset F^{\phi_*w}_{r+1}$.  Moreover, for some $s$ we have $F^w_s = \M_1$, so $F^{\phi_*w}_s = \M_2$ as the latter contains a basis.  For a basis $\B \subset \M_1$ and a point $w \in A_\B$, the component of $\phi_*w$ on $e \in \M_2 \setminus \phi \M_1$ is computed by taking the minimum weight appearing in the circuit expressing $e$ in terms of $\phi\B$. As a consequence, $\phi_*$ takes $A_\B$ piecewise-linearly isomorphically onto $A_{\phi \B}$. 
\end{proof}

\begin{proposition}\label{prop-extensionbundle}
Let $\Phi_1: |\Sigma| \to \Berg(\M_1)$ satisfy Definition \ref{def-matroid-vb-1}, and let $\phi: \M_1 \to \M_2$ be an extension, then we have the following:
\begin{itemize}
\item[(a)] $\Phi_2= \phi\circ \Phi_1$ satisfies Definition \ref{def-matroid-vb-1},
\item[(b)] The equivariant characteristic classes of $\Phi_1$ coincide with those of $\Phi_2$. 
\end{itemize}
\end{proposition}
\begin{proof}
The first statement is immediate by definition of $\phi_*$. For the second statement we observe that 
for any symmetric function $g$ with associated $g_\M: \Berg(\M) \to \R$ we have $g_{\M_1} = \phi_* \circ g_{\M_2}$. 
\end{proof}

\begin{definition}[Extension class of a tropical toric vector bundle I] \label{def-tmbec1}
Let $\Phi_1$ and $\Phi_2$ be as in Proposition \ref{prop-extensionbundle}, then we say that $\Phi_1$ is equivalent to $\Phi_2$.  
\end{definition}

\begin{definition}[Extension class of a tropical toric vector bundle II] \label{def-tmbec2}
We say that $\vv_1:\M_1 \to \O_{|\Sigma|}$ is equivalent to $\vv_2: \M_2 \to \O_{|\Sigma|}$ if the associated $\Phi_1: |\Sigma| \to \Berg(\M_1)$ is equivalent to $\Phi_2: |\Sigma| \to \Berg(M_2)$. We write $[(\M, \Phi)]$, $[\EE]$, or $[(\M, \vv)]$ for the equivalence class of a tropical toric vector bundle. 
\end{definition}

\subsection{Klyachko flats}
For any extension $\phi: \M_1 \to \M_2$ and flats $F_1, F_2 \subset \M_1$ we have:
$$ \rank(\phi_*(F_1 \cap F_2)) \leq \rank(\phi_*F_1 \cap \phi_* F_2).$$
It is possible that this inequality is strict. Taking this into account, the rank of the Klyachko flat associated to $\br \in \Z^r$ can be defined as the maximum of the ranks obtained on representatives in an extension class.  

\begin{definition}[DJS matroid of an extension class]\label{def-DJSmatroid}
    Fix a tropical toric vector bundle class $[\EE]$, then the matroid $\M$ of a representative $(\M, \Phi)$ is said to be a DiRocco-Jabbusch-Smith matroid for $[\EE]$ if  the ranks of the Klyachko flats in the tropical toric vector bundle determined by $(\M, \Phi)$ are the maximums of those obtained in the class, and $\M$ has no equivalent submatroid with this property. 
\end{definition}

In the representable case, the ranks of the Klyachko flats of a DJS matroid coincide with the vector space dimension of the corresponding subspaces of $E$.  As a consequence, the rank function becomes modular when restricted to the lattice generated by the Klyachko flats of a DJS matroid.  It is not clear that a DJS matroid always exists for a general tropical toric vector bundle $[\EE]$, and it is the case that for certain matroids, the resulting ranks of the Klyachko flats could not become modular (see Theorem \ref{thm-vamos} below).  We refer the reader to \cite{Hochstattler-Wilhelmi, Bonin}, where various aspects of extensions and the notion of matroid \emph{amalgams} are explored.   One approach to constructing a DJS matroid is to fix a representative $(\M, \Phi)$, and then for each $\br \in \Z^n$ find an extension $(\M_\br, \Phi_\br)$ which maximizes the rank of the corresponding Klyachko space.  The induced bundle on an amalgam $\bigcup_\M \M_\br$ would then simultaneously maximize all Klyachko ranks.  Unfortunately, amalgams do not always exist.  In fact, the main result of \cite{Bonin} shows that they always exist if and only if $\M$ is a \emph{modular} matroid. 

\begin{definition}\label{def-modular}
    The \emph{submodular defect} of two flats $F_1, F_2 \subset \M$ is defined to be:
    $$d(F_1, F_2) = \rank(F_1) + \rank(F_2) - \rank(\Span(F_1 \cup F_2)) - \rank(F_1 \cap F_2).$$
    A matroid $\M$ is said to be \emph{modular} if $d(F_1, F_2) = 0$ for every pair of flats $F_1, F_2 \subseteq \M$. 
\end{definition}

\begin{proposition}\label{prop-mod-ext-djs}
    Suppose $\M$ can be extended to a (possibly infinite) modular matroid $\mathcal{N}$, then for any $\Phi$, the tropical toric vector bundle class $[(\M, \Phi)]$ has a DJS matroid. 
\end{proposition}

\begin{proof}
    For all flats appearing in the Klyachko filtrations, the induced flats in $\mathcal{N}$ have submodular defect $0$. It follows that the ranks of the Klyachko flats are maximized in $\mathcal{N}$. 
\end{proof}
Observe that the conditions of Proposition \ref{prop-mod-ext-djs} hold for any representable matroid.  However, there are matroids which cannot be extended to a modular matroid; this is a consequence of the following proposition, which is ultimately a different perspective on (a special case of) \cite[Theorem 8]{Hochstattler-Wilhelmi}.

\begin{figure}[ht]
\begin{tikzpicture}
\begin{scope}

\draw[fill=red,opacity=0.13] (-4,0,-3) -- (-4,0,3) -- (4,0,3) -- (4,0,-3) -- cycle;
\draw[fill=red,opacity=0.1] (-4,0,-3) -- (-4,0,3) -- (0,3,3) -- (0,3,-3) -- cycle;
\draw[fill=red,opacity=0.1] (4,0,-3) -- (4,0,3) -- (0,3,3) -- (0,3,-3) -- cycle;
\draw[fill=red,opacity=0.1] (-4,0,-3) -- (-4,0,3) -- (0,-3,3) -- (0,-3,-3) -- cycle;
\draw[fill=red,opacity=0.1] (4,0,-3) -- (4,0,3) -- (0,-3,3) -- (0,-3,-3) -- cycle;

\filldraw[black] (-4,0,-3) circle (2pt) node[anchor=east]{$h_1$};
\filldraw[black] (-4,0,3) circle (2pt) node[anchor=east]{$h_2$};
\filldraw[black] (4,0,3) circle (2pt) node[anchor=west]{$f_2$};
\filldraw[black] (4,0,-3) circle (2pt) node[anchor=west]{$f_1$};

\filldraw[black] (0,3,3) circle (2pt) node[anchor=east]{$p$};
\filldraw[black] (0,-3,3) circle (2pt) node[anchor=east]{$q$};
\filldraw[black] (0,3,-3) circle (2pt) node[anchor=east]{$e$};
\filldraw[black] (0,-3,-3) circle (2pt) node[anchor=east]{$g$};

\end{scope}
\end{tikzpicture}
\caption{The Vamos matroid: shaded $4$-gons and any set of size greater than or equal to $5$ is dependent.} \label{fig-Vamos}
\end{figure}
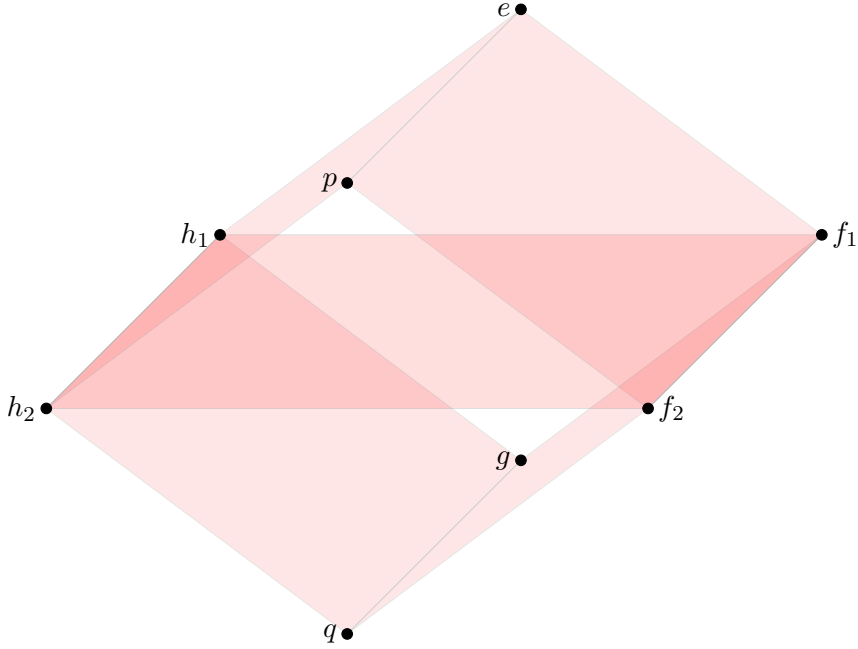

\begin{theorem}\label{thm-vamos}
    Let $V$ be the Vamos matroid.  Then there are flats $F, H \subset V$ such that for any extension $\phi: V \to \M$, the modular defect $\delta(\phi_*F, \phi_*H)$ is not $0$.
\end{theorem}

\begin{proof}
We immitate the proof of \cite[Theorem 8]{Hochstattler-Wilhelmi} while referring to Figure \ref{fig-Vamos}. We have $F = \{f_1, f_2\}$, $H = \{h_1, h_2\}$, $T_1 = H\cup\{e,p\}$, $T_2 = F \cup\{e,p\}$, $B_1 = H\cup\{g,q\}$, $B_2 = F \cup\{g,q\}$. We have $d(F,H) = \rank(F) + \rank(H) - \rank(F \cap H) - \rank(\Span(F \cup H)) = 2 + 2 - 0 - 3 = 1$. 

The defects of $T_1, T_2$ and $B_1, B_2$ are $\rank(T_1) + \rank(T_2) - \rank(T_1\cap T_2) - \rank(\Span(T_1 \cup T_2)) = 3 + 3 - 2 - 4 = 0$. 

Fix an extension $V \subset N$. For a flat $G \subset V$ we let $G_N$ denote the span of $G$ in $N$.  Suppose that $F_N$ and $H_N$ form a modular pair in $N$, then $\rank(F_N \cap H_N) = \rank(F_N) + \rank(H_N) - \rank(\Span(F_N \cup H_N)) = 2 + 2 - 3 = 1$. 

Let $D_1 = \Span(p, e)_N, D_2 = \Span(q, g)_N$, then $D_1 = (T_1 \cap T_2)_N = (T_1)_N \cap (T_2)_N$, and similarly $D_2 = (B_1)_N \cap (B_2)_N$. 

Moreover, $\rank(\Span(F_N \cap H_N \cup D_1)) = \rank((\Span F\cup D_1)_N \cap \Span(H \cup D_1)_N) = \rank((T_1)_N\cap(T_2)_N) = \rank(T_1 \cap T_2) = 2 = \rank(D_1)$.  This means $F_N \cap H_N \subset D_1$.  Similarly we have $F_N \cap H_N \subset D_2$. 

As a result, $F_N \cap H_N \subset D_1 \cap D_2$, and $1 \geq \rank(D_1 \cap D_2)$.

But also $\rank(\Span(D_1 \cup D_2)) = \rank(\Span(p,q,e,g)_N) = 4$, so $r(D_1) + r(D_2) = 4 < 4 +1 = \rank(\Span(D_1 \cup D_2)) + \rank(D_1 \cap D_2)$. This violates submodularity. 
\end{proof}

\begin{remark}\label{rem-coxunderextension}
For an extension $\phi: \M_1 \to \M_2$ and equivalent bundles $\Phi_1, \Phi_2 = \phi\circ \Phi_1$, there is a natural extension of both global section matroids and Cox matroids: $\phi_*: H^0(X_\Sigma, \EE_1) \to H^0(X_\Sigma, \EE_2)$, $\phi_*: M(\EE_1) \to M(\EE_2)$ induced by the corresponding extensions of Klyachko flats.   In particular, for a fixed $\br \in \Z^n$, the size the corresponding Klyachko flat could go up under an extension, as could the rank.   

With appropriate modularity assumptions (for example, assuming $\M_2$ can be extended to a possibly infinite modular matroid), one can find an extension where all ranks of Klyachko flats are simultaneously maximized.  More generally one could consider the maximum rank obtained over all extensions in the class. 
\end{remark}




\section{Splitting of tropical toric vector bundles and ampleness}\label{sec-splitting}
A toric vector bundle $\E$ is said to be \emph{split} if it is isomorphic to a direct sum of toric line bundles.  When phrased in terms of the data of a piecewise-linear function $\Phi_\E: |\Sigma| \to \tilde\B(E)$, the existence of a splitting is equivalent to the condition that the image of $\Phi_\E$ lies in a single apartment $\tilde A \subset \tilde\B(E)$. Equivalently, there is a basis $B \subset E$ such that the valuation $\vv_\E: E \to \PL(N, \Z)$ can be computed by the rule $\vv_\E(\sum C_i b_i) = \min\{\vv_\E(b_i) \mid C_i \neq 0\}$.  With these conditions in mind, we make the following definition. 

\begin{definition}\label{def-split}
    A tropical toric vector bundle $(\M, \Phi)$ is \emph{split} if the image of $\Phi$ lies in a single apartment of $\Berg(\M)$.  Equivalently, there is a basis $B \subset \M$ such that $\vv(e) = \min\{\vv(c) \mid c \in C\cap B\}$, where $C \subset \M$ is the unique circuit with $C \setminus \{e\} \subset B$. 
\end{definition}

Strictly speaking, a split tropical toric vector bundle is not isomorphic to a sum of toric line bundles.  Bringing in the notion of tropical toric vector bundle extension class allows us to remake the connection between these two concepts. 

\begin{proposition}[Splitting of a tropical toric vector bundle extension class]
The following are equivalent. 
\begin{itemize}
    \item[(a)] The class $[(\M, \vv)]$ contains a member of the form $(B, \vv)$, where $B$ is a single basis.
    \item[(b)] The class $[(\M, \vv)]$ contains the pair associated to a direct sum of toric line bundles. 
    \item[(c)] The class $[(\M, \vv)]$ contains a pair $(M', \vv')$ which is split. 
\end{itemize}
\end{proposition}
\begin{proof}
The pairs of the form $(B, \vv)$ as in $(a)$ are precisely the data of a direct sum of toric line bundles, so $(a)$ is equivalent to $(b)$.  Definition \ref{def-split} and the definition of extension then impies that $(a)$ is equivalent to $(c)$. 
\end{proof}



Starting with a tropical toric vector bundle $\EE$ corresponding to the data $(\M, \Phi)$, a splitting of $\EE$ is a chain of extensions of the form $\M \to \M_1 \leftarrow \cdots \leftarrow U^r_r$ or $\M \leftarrow \M_1 \to \cdots \leftarrow U^r_r$, compatible with corresponding piecewise linear maps $\Phi, \Phi_1, \ldots$.  One can think about the final entry in such a chain as a generalized apartment in $\Berg(\M)$. 

\subsection{Splitting over $\P^1$}

Over $\P^1$ the compatibility condition amounts to knowing that the two rays in $\Berg(\M)$ defined by $\Phi$ lie in one of these generalized apartments. That this can always be done for toric vector bundles is the combinatorial version of Grothendieck's famous splitting theorem.

\begin{corollary}  \label{cor-split-P1-modular}
    Suppose $\M$ can be extended to a (possibly infinite) modular matroid $\mathcal{N}$, then any bundle class $[(\M, \Phi)]$ over $\P^1$ splits. 
\end{corollary}

\begin{proof}
    It suffices to show that if $\mathcal{N}$ is a modular matroid then any pair of complete flags $\{F_i\}, \{G_j\}$ has a common adapted basis.  Let $r = rank(\mathcal{N})$.  For any flat $F$ of $\mathcal{N}$, the matroid induced by restricting the rank function of $\mathcal{N}$ to elements of $F$ is also modular. The statement clearly holds for any matroid of rank $1$, so we suppose that the statement holds for any modular matroid of rank $r-1$.  This means that we can find a basis $\B_0$ of $F_{r-1}$ adapted to both of the flags $F_{r-1} \supset \cdots \supset F_1$, $F_{r-1}\cap G_{r-1} \supset \cdots \supset F_{r-1} \cap G_1$ in $F_{r-1}$. We will extend $\B_0$ to a basis of $\mathcal{N}$ which is adapted to $G_{r-1} \supset \cdots \supset G_1$. 

    Let $1 \leq k \leq r-1$ be the first step where $\rank(F_{r-1}\cap G_k) < \rank(G_k)$.  Observe that $\rank(F_{r-1}\cap G_k) = \rank(F_{r-1} \cap G_{k-1}) = k-1$. Pick $b \in G_k \setminus F_{r-1}$, then $\B = \B_0 \cup\{b\}$ is a basis of $\mathcal{N}$; we claim $\B$ is adapted to all of $\{G_j\}$.  Let $m \geq k$.  The fact that $\mathcal{N}$ is modular implies:
    \[\rank(G_m \cap F_{r-1}) = m + (r-1) - r = m - 1.\]
    Now $\rank(G_m) = m$, and $b \in G_m \setminus F_{r-1}$, so we must have that $\B \cap G_m$ spans $G_m$.  
\end{proof}


From the discussion on submodular defects and the DJS matroid of a tropical toric vector bundle, one expects that the Vamos matroid provides interesting abberant behavior compared to the case of a modular matroid.

\begin{corollary}  \label{cor-tmb-P1-no-split}
    There is a bundle $(V, \Phi)$ over $\P^1$ which cannot be extended to a split bundle. 
\end{corollary}

\begin{proof}
    We design $(V, \Phi)$ so that the two rays of the fan of $\P^1$ go to the filtrations $V \supset F_i \supset \emptyset$ for $i = 1, 2$.  The induced bundle splits if and only if we can find an extension of $V$ with $\delta(\phi_*F_1, \phi_*F_2) = 0$, which contradicts Theorem \ref{thm-vamos}.
\end{proof}

We remark that any of the matroids constructed in \cite[Theorem 8]{Hochstattler-Wilhelmi} can also be used to provide examples.  In particular, for any matroid $\M$ with a flat $F$ and a hyperplane $H$ such that $F\cap H = \emptyset$ and $d(F, H) \neq 0$ there is an extension $\phi: \M \to \mathcal{N}$ such that there are bundles built on $\mathcal{N}$ over $\P^1$ with this property.  It is still possible that a member of the class of $(V, \Phi)$ obtained through extensions and "de-extensions" admits a splitting over $\P^1$.

\subsection{Ample and nef tropical toric vector bundles}
Now we turn our attention to one of the motivations for studying splitting over $\P^1$.  We would like to have a sensible definition of what it means for a tropical toric vector bundle to be ample or nef.  In \cite[Theorem 2.1]{HMP}, Hering, Musta\c{t}\u{a}, and Payne show that both ampleness and nefness for a toric vector bundle can be checked on the restrictions to toric curves.  

\begin{theorem}[Hering-Musta\c{t}\u{a}-Payne]
    Let $\E$ be a toric vector bundle over a smooth toric variety $X_\Sigma$ of dimension $d$, then $\E$ is ample (respectively nef) if and only if $\E|_{C_\tau}$ is ample (respectively nef) for all irreducible toric curves $C_\tau$ for $\tau \in \Sigma(d-1)$. 
\end{theorem}

\begin{remark}
In \cite[Section 2.2]{Kaveh-Manon-TVBs-valuation}, the above characterization of ample and nef toric vector bundles is rephrased in terms of convexity properties of piecewise-linear maps to an (extended) spherical building. 
\end{remark}

Any restriction $\E|_{C_\tau}$, one must consider to apply this criterion, can be regarded as a toric vector bundle over $\P^1$.  In particular, $\E|_{C_\tau} \cong \bigoplus_{i = 1}^r\O(n_i)$.  Then, $\E|_{C_\tau}$ is ample (respectively nef) if and only if $n_i > 0$ (respectively $\geq 0$) for all $1 \leq i \leq r$.  As this criterion implicitly assumes that $\E|_{C_\tau}$ splits, it is not immediately clear that it generalizes to tropical toric vector bundles.  If every restriction of a tropical toric vector bundle to the toric curves is equivalent to a split bundle over $\P^1$, then these conditions do make sense.  This occurs for tautological bundles which is expected as we have also shown them to be globally generated (Corollary \ref{cor-tautologicalGG}). 

\begin{definition}[Ample and nef tropical toric vector bundle]  \label{def-ample-nef-tmb}
We say that a tropical toric vector bundle $\EE$ on a smooth toric variety $X_\Sigma$ is \emph{ample} (respectively \emph{nef}) if its restriction to each toric curve $C_\tau$ in $X_\Sigma$ splits and is ample (respectively nef).    
\end{definition}

\begin{theorem}\label{thm-neftautological}
    Let $\M$ be a matroid with $m$ elements, and let $\EE_\M$ be tautological bundle over $X_m$, then $\EE_\M$ is nef in the sense that for all toric curves $C_\tau \subset X_m$, the restriction $\EE_\M|_{C_\tau}$ is equivalent to the bundle $\O(1)\oplus \O_{\P^1} \oplus \ldots \oplus \O_{\P^1}$ or $\O_{\P^1} \oplus \ldots \oplus \O_{\P^1}$ over $\P^1$. 
\end{theorem}
\begin{proof}
    Let $\sigma,\sigma' \in \Sigma_m(m)$ be faces with $\tau = \sigma \cap \sigma' \in \Sigma_m(m-1)$. Using the symmetric group action, we can assume that $\sigma$ corresponds to the standard ordering on $[m]$, and that $\sigma'$ corresponds to the ordering where $i$ and $i+1$ are exchanged for some $i \in [m]$. The greedy bases $\B_\sigma$ and $\B_{\sigma'}$ are either the same basis, or they differ by exchanging $\e_i$ with $\e_{i+1}$.   The restriction $\EE_\M|_{C_\tau}$ is then equivalent to a tropical toric vector bundle defined on the matroid induced on $\B_\sigma \cup \B_{\sigma'}$.  The characters corresponding to the elements of this matroid on either face of the fan of $\P^1$ are induced by the indicator vectors of the elements. In the case that $\B_\sigma = \B_{\sigma'}$, the bundle is already split, and the characters associated to each element agree on either face of the fan of $\P^1$. If $\B_\sigma$ and $\B_{\sigma'}$ differ, the resulting matroid is of the form $\{b_1, \ldots, b_{r-1},x, y\}$, where a single circuit links $x, y$, and a collection of the $b_j$.  Any such matroid is representable, and so  $\EE_\M|_{C_\tau}$ is equivalent to a bundle on a basis $\{b_1, \ldots, b_{r-1}, z\}$ for some an element $z$. The characters for $b_1, \ldots, b_{r-1}$ are induced from the indicator vector of some $\e_j$, $j \neq i, i +1$ on either face of the fan of $\P^1$, while $z$ has weights induced from $\delta_{i}, \delta_{i+1}$ on the two rays. The former correspond to copies of $\O_{\P^1}$. For the latter case, we assume that $\e_i \in \B_\sigma$, then the induced summand is $\O(a)$, where $\delta_{i} - \delta_{i+1} = a \eta_\tau$ for $\eta_\tau$ the primitive generator of $\tau^\perp$ which is positive on $\sigma$.  The character $\eta_\tau$ is $\delta_i - \delta_{i+1}$, so $a = 1$.
    \end{proof}

\end{document}